\documentclass[a4paper,12pt]{article}
\usepackage[utf8]{inputenc}
\usepackage[english]{babel}
\usepackage{amssymb}
\usepackage{amsthm}
\newtheorem{theorem}{Theorem}
\newtheorem{lemma}[theorem]{Lemma}
\newtheorem{corollary}[theorem]{Corollary}

\usepackage{amsmath}
\usepackage{graphicx}
\graphicspath{{./Figs/}{.}}

\usepackage{tabulary}
\usepackage{multirow}

\renewcommand{\epsilon}{\varepsilon}
\newcommand{\cA}{A}
\newcommand{\cB}{B}
\newcommand{\cC}{C}
\newcommand{\cD}{D}
\newcommand{\pP}{\mathcal{P}}
\newcommand{\bB}{\mathcal{B}}
\newcommand{\yY}{\mathcal{Y}}
\newcommand{\vV}{\mathcal{V}}
\newcommand{\jJ}{\mathcal{J}}
\newcommand{\eps}{\varepsilon}
\newcommand{\beq}{\begin{equation}}
\newcommand{\eeq}{\end{equation}}

\newcommand{\Rset}{{\mathbb R}}
\newcommand{\eq}[1]{(\ref{#1})}
\newcommand{\pain}{Painlev\'{e} }

\begin{document}

\title{Dynamics beyond dynamic jam; unfolding the Painlev\'{e} 
paradox singularity}

\author{Arne Nordmark, P\'eter L. V\'arkonyi, Alan R. Champneys} 
\date{10th May 2017}
\maketitle

\begin{abstract}

This paper analyses in detail the dynamics in a neighbourhood of a
G\'{e}not-Brogliato point, colloquially termed the G-spot, which
physically represents so-called dynamic jam in rigid body mechanics
with unilateral contact and Coulomb friction. Such singular points
arise in planar rigid body problems with slipping point contacts at the
intersection between the conditions for onset of 
lift-off and for the Painlev\'{e} paradox.  The G-spot can be
approached in finite time by an open set of initial conditions in a
general class of problems.  The key question addressed is what happens
next.  In principle trajectories could, at least instantaneously, lift
off, continue in slip, or undergo a so-called impact without
collision. Such impacts are non-local in
momentum space and depend on properties evaluated away from the
G-spot. 

The answer is obtained via an analysis that involves a consistent
contact regularisation with a stiffness proportional to
$1/\varepsilon^2$.  Taking a singular limit as $\varepsilon \to 0$,
one finds an inner and an outer asymptotic zone in the neighbourhood
of the G-spot.  Matched asymptotic analysis then enables not just the
answer to the question of continuation from the G-spot in the limit
$\varepsilon \to 0$ but also reveals the sensitivity of trajectories
to $\varepsilon$. The solution involves large-time asymptotics of
certain generalised hypergeometric functions, which leads to
conditions for the existence of a distinguished smoothest trajectory
that remains uniformly bounded in $t$ and $\varepsilon$.  Such a
solution corresponds to a canard that connects stable slipping motion to
unstable slipping motion, through the G-spot. Perturbations to the
distinguished trajectory are then studied asymptotically. 

Two distinct cases are found according to whether the contact force
becomes infinite or remains finite as the G-spot is approached. In the
former case it is argued that there can be no such canards and so an
impact without collision must occur.  
In the latter case, the canard trajectory acts as a dividing
surface between trajectories that momentarily lift off and those that
do not before taking the impact. The orientation of the initial
condition set leading to each eventuality is shown to change each time
a certain positive parameter $\beta$ passes through an integer.

Finally, the results are illustrated on a particular physical example,
namely the a frictional impact oscillator first studied by Leine {\em
  et al.}
\end{abstract}

\section{Introduction}
\label{sec:1}
This paper considers the open question first posed in the work of
G\'{e}not and Brogliato \cite{Genot1999} in relation to the classical
Painlev\'{e} paradox in contact mechanics. They considered the
classical problem of a falling rod one end of which is in contact with
a rough horizontal surface (see Fig.~\ref{fig:examples}(a)).  They
show that for sufficiently high coefficient of friction there is an
open set of initial conditions that are drawn in finite time into a
singularity, which we have termed the G-spot in homage to G\'{e}not.
Such a point both is characterised by the vanishing of normal free
acceleration and the so-called Painlev\'{e} parameter which measures
the ratio of that acceleration to normal contact force. The physical
phenomenon of approaching the singularity is also known as {\em
  dynamic jam}\/ \cite{Or2012} and has been reported in other physical
systems.  In particular in Sec.~\ref{sec:6} below, the results of this
paper shall be applied to the frictional impact oscillator system
represented in Fig.~\ref{fig:examples}(b), first studied by Leine {\em
  et al.}\/ \cite{Leine2002}.  Yet, we are unaware of any mathematical
analysis of what must happen after such a singularity is reached. Does
the rigid body formulation break down completely, so that there is no
continuation of trajectories beyond this point? If so, then can we say
what might happen physically?

\begin{figure}
\begin{center}
\includegraphics[width=0.8\textwidth]{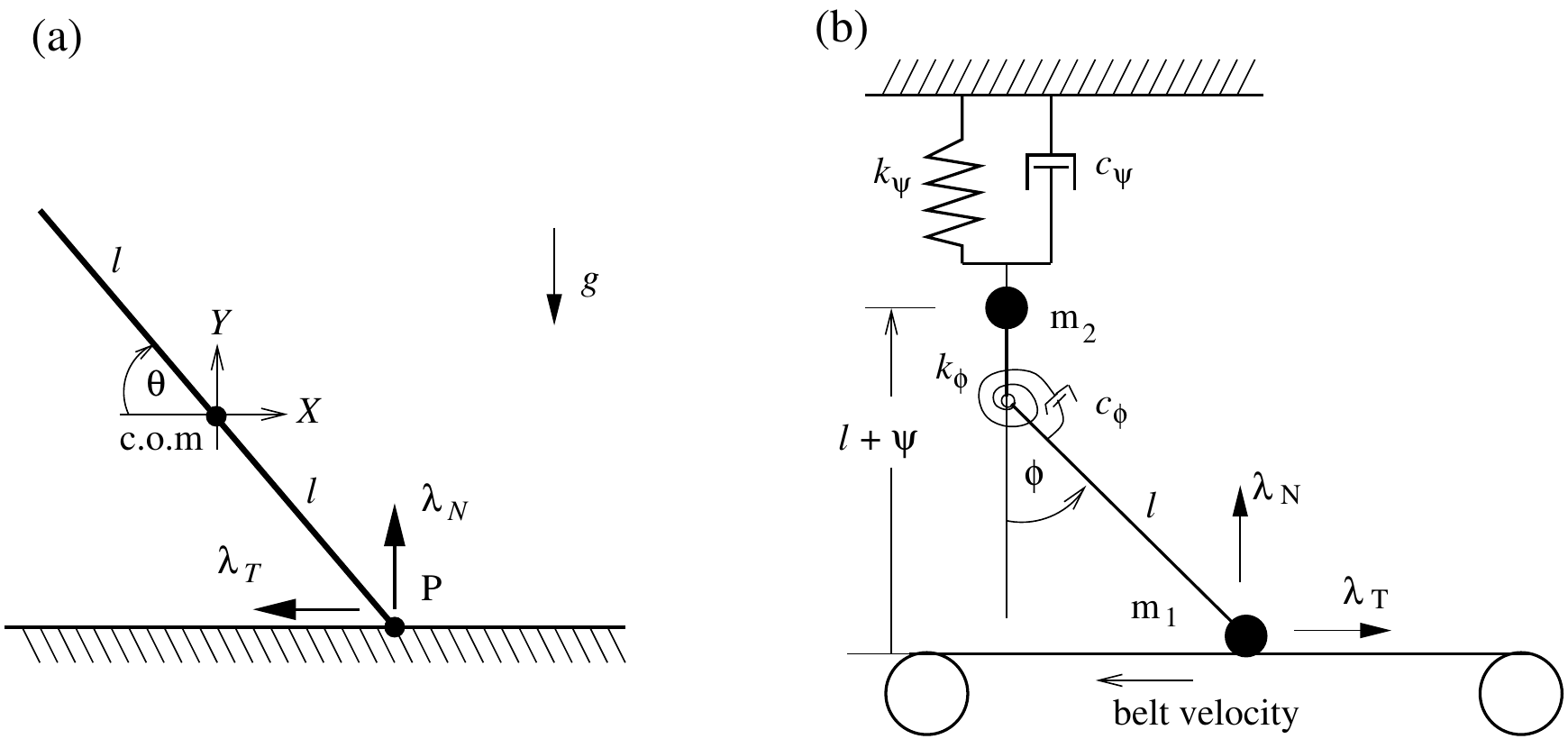}
\end{center}
\caption{(a) The canonical example of the 
Painlev\'{e} paradox, a rod falling under gravity. (b) 
The frictional impact oscillator proposed by \cite{Leine2002}, that we
shall return to in Sec.~\ref{sec:6}.} 
\label{fig:examples}
\end{figure}

\subsection{The Painlev\'{e} paradox and impact without collision}

Our work follows the formalism and notation introduced in the recent
review paper by two of us \cite{PainleveReview}, to which we refer the
reader for the necessary motivation, historical context and general
formulation. In particular, there it was argued that approach to a
G-spot singularity is a generic mechanism in planar rigid body
mechanics subject to unilateral point contact with dry frictional
surfaces and is not merely restricted to a few atypical example
mechanisms.

Specifically, we consider a multi-degree-of-freedom
Lagrangian planar rigid body system with an isolated point of contact 
with a rigid surface, which is subject to Coulomb friction. 
Using the notation introduced in \cite{PainleveReview}, we find that
projecting the Lagrangian equations of motion onto 
tangential and normal directions gives scalar equations  
\begin{align}
\dot{u}& =a\left( q,\dot{q},t\right) +\lambda_{T}\cA\left( q,t\right)
+\lambda_{N}\cB\left( q,t\right) ,  \label{eq:udynamics} \\
\dot{v}& =b\left( q,\dot{q},t\right) +\lambda_{T}\cB\left( q,t\right)
+\lambda_{N}\cC\left( q,t\right).  \label{eq:vdynamics}
\end{align}
Here $u$ and $v$ are tangential and normal velocities of the contact point; $q$ is a vector of generalised coordinates and $t$ is time;  $a$, $b$, $\cA$, $\cB$, $\cC$, $\cD$ are scalars
subject to the constraints that $\cA>0$, $C>0$ and $\cA\cC-\cB^2 >0$ 
which arise
from the assumption of positive definiteness of the mass matrix.    
The scalars $\lambda_N\geq 0$ and $\lambda_T$ represent 
normal and tangential contact forces respectively and are Lagrange multipliers that must
be solved for under different assumptions on the mode of motion (free, stick, 
or positive or negative slip). During contact, we suppose that 
Coulomb friction applies:
\beq
|\lambda_T| \leq \mu |\lambda_N|, \qquad 
 \lambda_T =  - \mu \: \mbox{sign}(u) \lambda_N  \quad \mbox{for} \: u \neq 0,
\label{eq:Coulomb}
\eeq
where $\mu$ is the coefficient of friction.  

Positive slip occurs during contact with 
$\lambda_N>0$ and $u > 0$, so that $\lambda_T=-\mu\lambda_N$. To sustain
contact we must have $\dot{v}=0$, which gives 
\begin{equation}
  \lambda_N=-\frac{b}{p}, \qquad  \mbox{where} \quad  p := \cC-\mu \cB.
\label{eq:pain_par+}
\end{equation}
Here we dropped the superscript $+$ on the Painlev\'{e} parameter $p$, 
adopted in \cite{PainleveReview},
because this paper shall, without loss of generality, 
only concern positive slip. 
If $p$ is negative, we say the Painlev\'{e} paradox applies for
appropriate initial conditions, in which case 
\eqref{eq:pain_par+} shows that in
order for $\lambda_N$ to be positive we much have free normal acceleration
away from the contact, that is $b>0$. This leads to 
multiplicity of solution, because for $b>0$ lift-off could also occur.
However, it can be showed that slipping in regions with $b>0$ is violently
unstable \cite{paper2,HoganKristiansen}.
Nevertheless, there is another possibility, and indeed this is
the only consistent possibility if $p<0$ and $b<0$, namely that a so-called 
{\em impact without collision} (IWC) occurs, 
see \cite{paper1,PainleveReview}.

Impact  in the present context defines a process in 
which rapid changes in normal and tangential velocity occur over
an infinitesimal timescale \cite{Stronge2000}.
The impact process can then be modelled as a composite mapping
from an incoming velocity to an outgoing one:
\begin{equation}
(u^-,v^-) \mapsto_{\mbox{compression phase}} (u^{(0)},0) 
\mapsto_{\mbox{restitution phase}} (u^+,v^+), 
\label{eq:impactmap}
\end{equation}
where $v^-\leq 0$ and $v^+ \geq 0$. In  
each of the compression and restitution phases, it is assumed that the system behaves as a rigid body system (despite the presence of large forces), and one needs to account
for possible transitions from slip to stick.
Complete results are summarised in  \cite{paper1}, for an energetic
coefficient of restitution  where
the work done by the normal force during restitution is
$-r^2$ that lost during compression.
Similar calculations can be carried out explicitly
for a Poisson impact law where the normal impulse in restitution is 
$-r$ times that in compression (see e.g.~\cite{Chatterjee1998}). 
The distinction is not important here.  During the impact process, 
to leading order,
the motion can be assumed to occur along straight lines in
the $(u,v)$-plane, with corners occurring at transitions between slip
and stick during the impact process; see Fig.~\ref{fig:impactmap} for
two examples. 

\begin{figure}
\begin{center}
\includegraphics[width=0.6\textwidth]{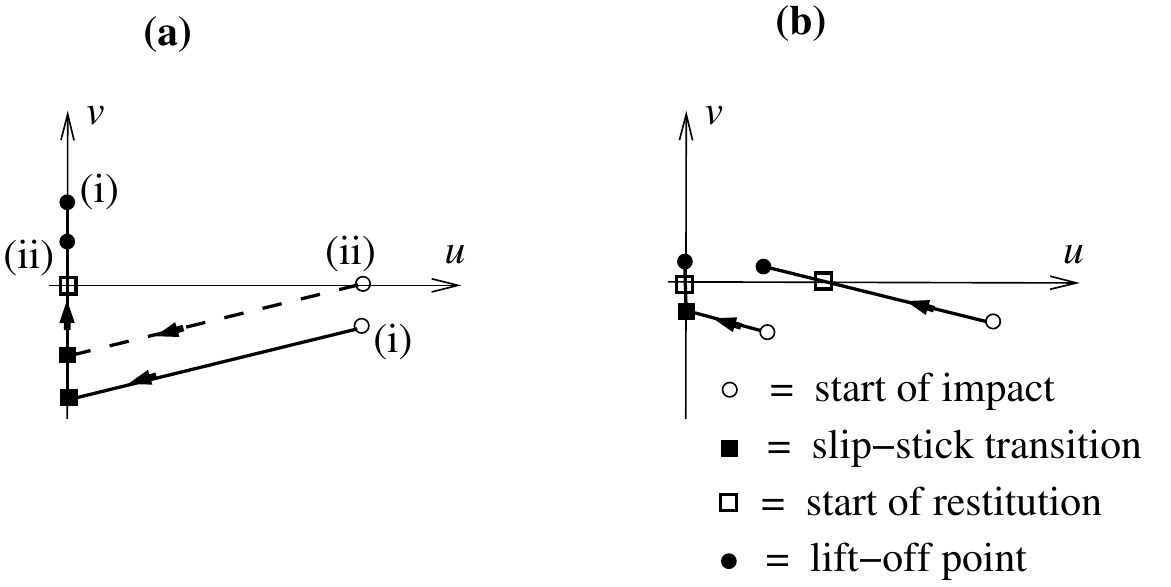}
\end{center}
\caption{(a) Representing trajectories during an impact for an initial
condition $u>0$ in the case that $p<0$ is small and constant.  The
solid line indicates a trajectory (labelled (i)) with initial $v<0$,
whereas the dashed line indicates a trajectory (ii) undergoing 
impact without collision, in which the initial condition has
$v=0$.  (b) Similar Impact events for $p$ small and positive. Here
one initial condition lifts off in slip, the other (with smaller
initial $u$) lifts off in stick.  See \cite{paper1} for details.}
\label{fig:impactmap}
\end{figure}

Looking at Fig.~\ref{fig:impactmap}(a), note that in the case that
$p<0$ then we can have an impact even when $v^-=0$. This would be an
example of an IWC, also known as a
tangential shock.

\subsection{The G-spot}

A question at the heart of the classical \pain paradox is what happens
next when a configuration with $p=0$ is reached during regular
slipping motion. As first shown by G\'{e}not and Brogliato \cite{Genot1999}
for the classical Painlev\'{e} paradox (see \cite{PainleveReview} for a 
generalisation), in fact the only
possible way to approach such a transition is via the codimension-two
point in phase space where $p=b=0$, i.e.~the G-spot.  They analysed nearby
trajectories by introducing a singular rescaling of time
\begin{equation}
dt=p \: d\hat{s} \label{eq:time_rescale},
\end{equation}
so that the $G$-spot becomes an equilibrium point in suitable variables that
evolve on the timescale $\hat{s}$.  In so doing, 
(see Sec.~\ref{sec:2} below for details) we obtain a system of the form
\begin{eqnarray}
\frac{d}{d\hat{s}} p &=& \alpha_{1} p  , \label{eq:Gspot1} \\
\frac{d}{d\hat{s}} b &=&\alpha_{2} p +\alpha_{3}b, \label{eq:Gspot2} 
\end{eqnarray}
where the $\alpha_i$ are constants to leading order, 
and, depending on the particular system in question can 
take on any combination of signs. 

The dynamics of \eqref{eq:Gspot1},\eqref{eq:Gspot2} can be analysed
using phase plane analysis.  Given an initial condition in slip $(b<0,
p>0)$ then there are only three possible outcomes. Either (i) the
trajectory remains in slip leaving the vicinity of the G-spot without
$b$ or $p$ changing sign; (ii) it lifts off by passing through $b=0$
$p>0$, or (iii) it is attracted to the G-spot $p=b=0$ as $\hat{s} \to
\infty$. Note though that $\hat{s} \to \infty$ implies that the G-spot is
approached in finite time $t$.

It is straightforward to show that this 
third possibility can only 
occur under specific sign combinations of 
$\alpha_1$, $\alpha_2$ and $\alpha_3$ see \cite[Fig.~13]{PainleveReview}. The
three relevant cases are summarised in Fig.~\ref{fig:Gspot3}:
\begin{description}
\item[Case I.] If $\alpha_2<0$ and $\alpha_1<\alpha_3<0$ then all initial conditions 
in the bottom right $(p,b)$-quadrant approach the $G$-spot
such that ratio $p/b \to 0$ and the  normal force $\lambda_N \to \infty$. 
\item[Case II.] If $\alpha_2>0$ and $\alpha_1<\alpha_3<0$ then initial conditions with 
$\alpha_2 p < (\alpha_1-\alpha_3 ) b $ similarly approach the $G$-spot with
 $p/b \to 0$ and $\lambda_N \to \infty$. 
\item[Case III.] If $\alpha_2<0$ and $\alpha_3<\alpha_1<0$ then the G-spot is
approached tangent to the nontrivial eigenvector 
$\alpha_2 p = (\alpha_1-\alpha_3 ) b $ 
and $\lambda_N$ approaches the finite limit $\alpha_2/(\alpha_3-\alpha_1)$. 
\end{description}

\begin{figure}
\begin{center}
\includegraphics[width=0.8\textwidth]{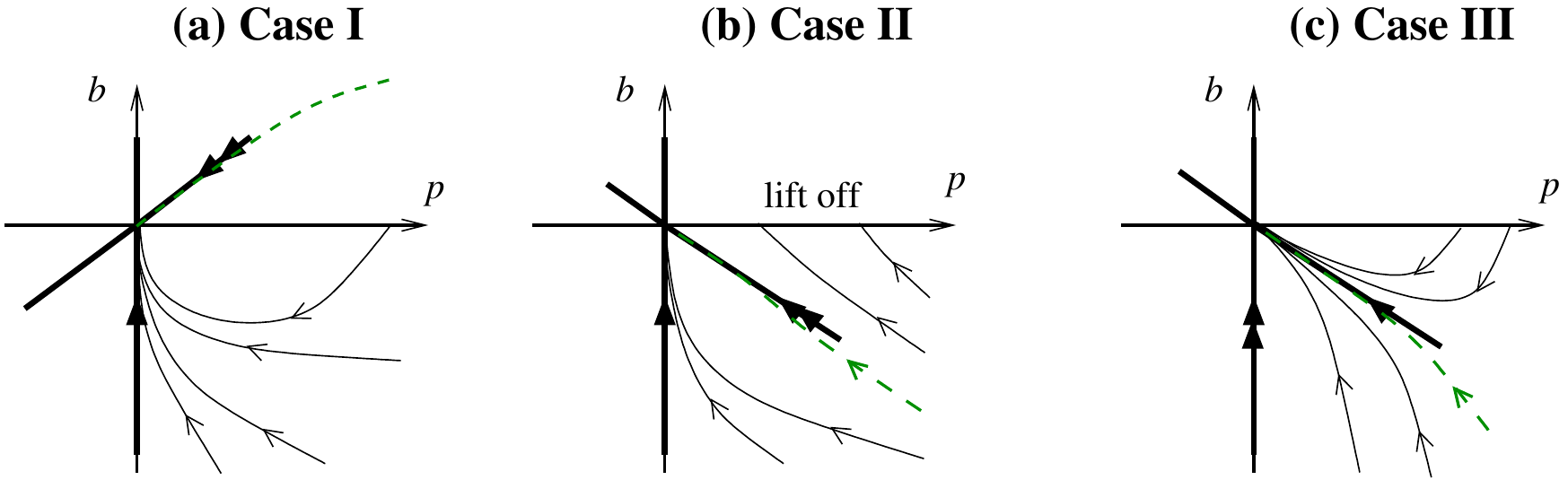}
\end{center}
\caption{Qualitative illustration of the dynamics of the singular
  system \eqref{eq:Gspot1}--\eqref{eq:Gspot2} in the $(p,b)$ phase
  plane in each of cases I, II and III with each of the $\alpha$'s
  assumed to be constant.  Here bold lines depict eigenvectors, with
  double arrows showing the strong stable eigendirection and single
  arrows the weak stable direction. Thin lines indicate individual
  trajectories in positive slip which can either be seen to lift-off
  (by reaching the positive $p$-axis) or to undergo dynamic jam (by
  reaching the G-spot $b=p=0$. The dashed line (green online) represents
  the distingished maximally smooth trajectory (see
  Sec.~\ref{sec:distinguished} below).}
\label{fig:Gspot3} 
\end{figure}

Note that the calculation in \cite{Genot1999} reveals that the
classical falling rod problem is in Case II. In what follows we shall
introduce examples of all three cases.

\subsection{Continuation beyond the G-spot and contact regularisation}

The rest of this paper concerns what happens for trajectories
that pass through G-spot. As we shall see, even that question cannot be answered
in complete generality. Our approach though is to 
to introduce a different scaling than \eqref{eq:time_rescale}
that is singular at $p=0$. In principle, a trajectory passing through
the G-spot could continue in (highly unstable) slip with $b>0$, it
could lift off, or it could take an IWC (see
Fig.~\ref{fig:qual}(a)).  However, there is a subtle problem with this
latter possibility, as illustrated in Fig.~\ref{fig:qual}(b). Detailed
calculations of the impact map (see \cite{paper1}) reveal that slope
of the impacting trajectory in the $(u,v)$-plane is proportional to
$p$.  But, precisely at the G-spot we know that $p=0$, so one has to
consider a different scaling than that used in \cite{paper1} to 
calculate the impact map. In principle, 
as illustrated in Fig.~\ref{fig:qual}(b), the
curvature of the impacting trajectory in the $(u,v)$-plane could be
such that the impact could terminate at any $u$-value between $u^-$
and 0. In other words, we cannot tell {\em a priori}\/ whether the
impact continues all the way to stick $u=0$, or whether it terminates
while the contact is still is slip ($u>0$).

The approach we shall take to addressing these questions is to study the
problem via contact regularisation; that is, replacing the rigid
constraint with a compliant one whose stiffness scales like some small
parameter $\eps^2$, see \cite{PainleveReview} and references therein.
In particular, in \cite{paper2} the idea was introduced of finding
resolutions to certain indeterminate cases of the \pain paradox via
such an approach and taking the limit $\eps \to 0$.  If there is a
unique solution that can be followed uniformly into this limit, then
it was said those dynamics are said to be \emph{uniformly resolvable}.
This enables, for example, questions to be answered of whether slip
with $p<0$ could be stably observed in practice 
(it can't, it is wildly unstable).
This is precisely the approach we shall take here.  The key will be to
find a consistent asymptotic scaling that enables us to identify a
distinguished trajectory that is smooth both in $t$ and $\eps$ in a
neighbourhood of the G-spot. Such a maximally smooth trajectory is
illustrated as a dashed red line in Figs.~\ref{fig:Gspot3} and
\ref{fig:qual}. We then consider perturbations to this trajectory to
decide whether nearby trajectories take an impact 
or lift off, and over what timescale.

\subsection{Summary of main results}

The main result of the paper is to perform a matched asymptotic expansion that
enables a description of what happens beyond the G-spot for a general
planar rigid body system with an isolated frictional contact point. This is 
achieved by finding a dinstinguished inner asmyptotic scale for $p$ (or equivalently time) of 
size $\mathcal{O} \eps^{2/3}$ where $\eps^{-2}$ is the regularised
contact stiffness. A matched asymptotic analysis then  
leads to regular conclusions
as $\eps \to 0$.  Fig.~\ref{fig:qual} depicts a qualitative 
representation of the results for small $\eps>0$.  
Specifically, we find: 
\begin{description}
\item[Cases I and II.] As depicted in Fig.~\ref{fig:qual}(e),
all trajectories that pass through a small neighbourhood
of the G-spot take an IWC. The process of what happens after this impact
cannot be analysed by studying the dynamics of a 
neighbourhood of the G-spot alone, because
the impact necessarily involves $O(1)$ changes in 
$u$ and $v$ (as shown in Fig.~\ref{fig:qual}(b)) and could involve lift off
with zero or finite tangential velocity $u$, depending on the precise example
system in question. 

\item[Case III.]  In this case, there is a distinguished trajectory, indicated by a dashed (red) line in Fig.~\ref{fig:qual}(c),(d) that forms
a {\em canard} solution for small $\eps>0$ which divides two different 
behaviours. On one side of this distinguished trajectory, solutions lift off,
whereas on the other side, they take an IWC. Each time the ratio 
$\beta=\alpha_1/\alpha_3$ evaluated at the G-spot 
passes through a positive integer, then there is side-switching between which
sign of perturbation to the canard undergoes lift-off and which undergoes
IWC. 
\end{description}
More precise statements of these results are given in Secs.~\ref{sec:4} 
and \ref{sec:5} below. 

\begin{figure}
\begin{center}
\includegraphics[width=\textwidth]{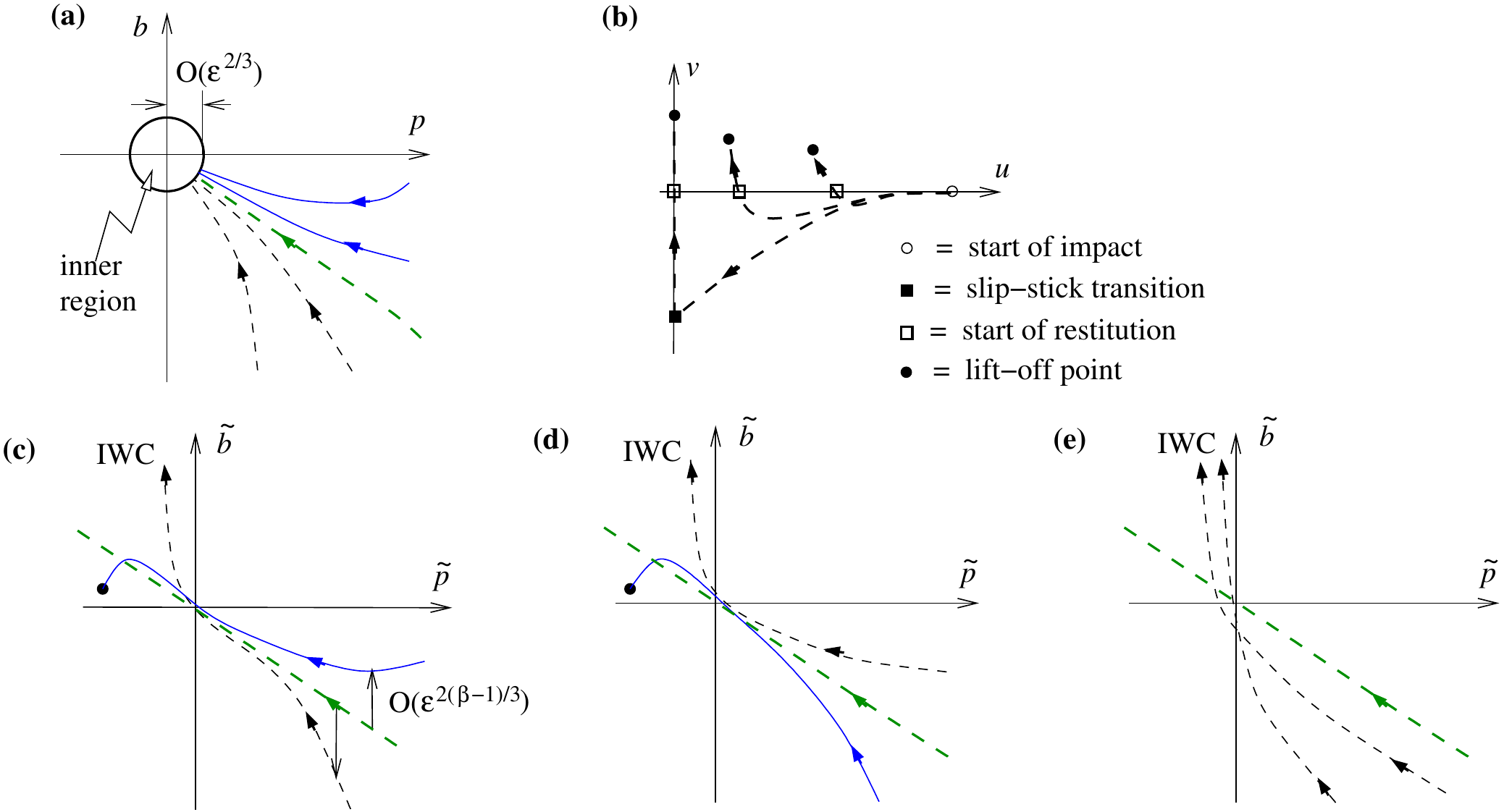}
\caption{Qualitative representation of the main results. (a) Indicating 
the scaling of the inner region. Specifically illustrated here is an
example of Case III from Fig.~\ref{fig:Gspot3}. In this and subsequent panels,
a thick dashed (green online) line is used to represent the 
distinguished trajectory, thin solid lines (blue online) 
represent trajectories that 
lift off after passing through a neighbourhood of the G-spot and 
thin dashed lines  represent trajectories that take an impact without collision. 
(b) Indicating the process of impact without collision in the (outer) 
$(u,v)$ co-ordinates in a neighbourhood of the G-spot. 
Depending on the global dynamics $u$-direction the trajectory can either
continue all the way down to stick, or can curve upwards and lift-off while
still slipping. (c-e) Representation of dynamics of perturbations to the 
distinguished trajectory in the inner region: (c) Case III when $[\beta]$ is an odd integer; (d)  Case III when $[\beta]$ is
even integer; and (e) Case II. (Here $[\cdot]$ represents the integer part of 
positive number) }
\label{fig:qual}
\end{center}
\end{figure}

\subsection{Outline}

The rest of this paper is outlined as follows.  Section \ref{sec:2}
below introduces a general formulation that includes constraint
regularisation via an additional degree of freedom, that can be
thought of as an additional spring. We also introduce a simple
illustrative example system in which many of the calculations can be
done explicitly. Sections \ref{sec:3} and \ref{sec:4} then contain
numerical and analytical calculations respectively on this example
in order to motivate and illustrate the general principles. 
Section \ref{sec:5} then
contains a detailed asymptotic derivation of the main results of the
paper for arbitrary planar $m$-dimensional rigid-body system
with a single frictional point contact.  Section \ref{sec:6} then
contains application of the results to the frictional impact
oscillator. Finally, Sec.~\ref{sec:7} draws conclusions and suggests
avenues for other work.

\section{Preliminaries}
\label{sec:2}

Consider a planar Lagrangian system with a unilateral constraint, 
which can be expressed in the form  $y>0$, where 
$y$ is a smooth function of the co-ordinate
variables $q_i$.
To simplify notation, in what follows we group together all Lagrangian
co-ordinates and velocities variables $q_i$, $\dot{q}_i$ and (in the
case of explicitly non-autonomous systems) $t$ into a single
$m$-dimensional state vector $\xi$ and consider systems that
are written in the form
\begin{equation}
\dot{\xi} = F(\xi) + G_T(\xi) \lambda_T+G_N(\xi) \lambda_N,\label{eq:genTN},
\end{equation}
where the scalar Lagrange multipliers $\lambda_N$ and $\lambda_T$ 
represent the normal and tangential forces at the contact point. 
We also suppose that tangential and normal contact coordinates 
$x(\xi)$ and $y(\xi)$ are smooth functions of $\xi$. 
Now we can express the quantities entering in the normal and tangential
equations \eqref{eq:udynamics} and
\eqref{eq:vdynamics} as
\begin{equation*}
  u=\mathsterling_F x, \quad v=\mathsterling_F y, \quad
  a=\mathsterling_F u, \quad b=\mathsterling_F v, \quad A=\mathsterling_{G_T} u, \quad B=\mathsterling_{G_N} u=\mathsterling_{G_T} v, \quad C=\mathsterling_{G_N} v 
\end{equation*}
where $\mathsterling$ denotes the {\em Lie derivative}. In this context,
the Lie derivative $\mathsterling_G z$ of a scalar function $z(\xi)$ with respect to a vector field $G(\xi)$ is just
the total time derivative of $z$ under the assumption that $\xi$ satisfies
the dynamical system $\dot{\xi}=G(\xi)$. 
Additionally, the Lagrangian character of the system requires
\begin{equation}
  \mathsterling_{G_T} x=\mathsterling_{G_N} x=\mathsterling_{G_T} y=\mathsterling_{G_N} y=0
  \label{eq:lagrangian1}
\end{equation}
since $x$ and $y$ do not depend on $\dot{q}$.

If we restrict attention to positive slip, 
where $\lambda_T=-\mu\lambda_N$, we obtain
\begin{equation}
\dot{\xi} = F(\xi) + G(\xi) \lambda_N,
\label{eq:genx}
\end{equation}
where $G=G_N-\mu G_T$. We can now further define
\begin{equation*}
  p=C-\mu B=\mathsterling_G v, \quad \alpha_1=\mathsterling_F p, \quad
  \alpha_2=\mathsterling_F b, \quad \alpha_3=-\mathsterling_G b
\end{equation*}
and again the Lagrangian character requires
\begin{equation}
  \mathsterling_G p=0.
   \label{eq:lagrangian2}
\end{equation}
Under these definitions, in positive slip 
the scalar quantities $p(\xi)$, 
$b(\xi)$, $y(\xi)$, $v(\xi)$ satisfy
\begin{align}
  \dot{p}&=\alpha_1(\xi), 
\label{eq:gensys2} 
\\
  \dot{b}&=\alpha_2(\xi)-\alpha_3(\xi)\lambda_N,
\label{eq:gensys3} 
\\
  \dot{y}&=v,
\label{eq:gensys4} 
\\
  \dot{v}&=b+p\lambda_N.
\label{eq:gensys5} 
\end{align}

In what follows will denote by $\xi^*$ a point that satisfies the conditions 
to be at a G-spot in positive slip
\begin{equation*}
  p(\xi^*)=b(\xi^*)=y(\xi^*)=v(\xi^*)=0,
\end{equation*}
and use an asterisk to denote functions evaluated at such a point.

\subsection{Regularlised contact motion}

In the rigid limit, the constraint surface is given by 
$y(\xi)=0$. Following the approach outlined in the introduction, we 
shall analyse the system by introducing a regularisation in the form of 
a smoothing of the contact motion. Here we introduce  
compliance via an additional degree of freedom with
co-ordinate $z$ that represents the vertical deformation of the surface.
Then the normal force $\lambda_N$ becomes a function of $z$ and the vertical
position $y$ of the contact point of the rigid body( Fig.~\ref{fig:compliance})

\begin{figure}
\begin{center}
(a) \hspace{6.0cm} (b) \\
\includegraphics[width=40mm]{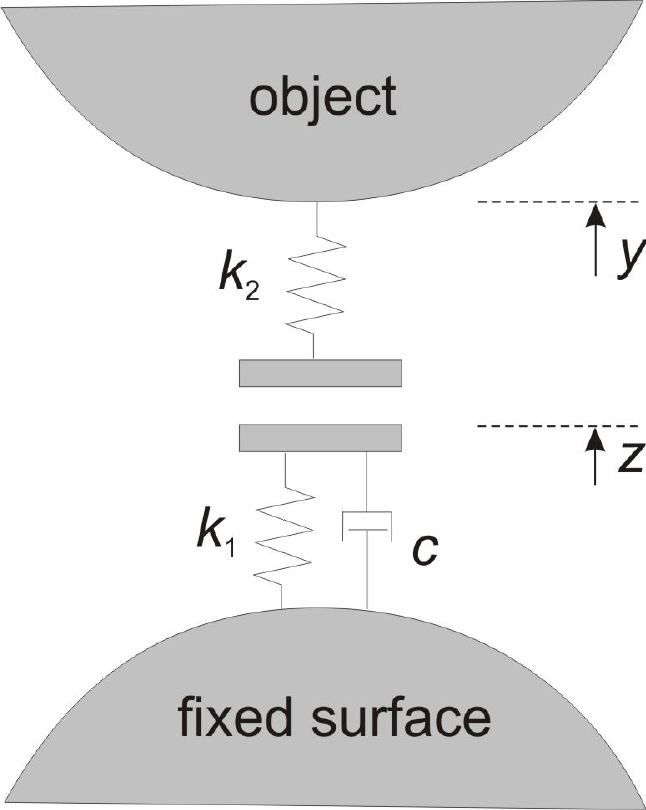}
\hspace{1.0cm}
\includegraphics[width=80mm]{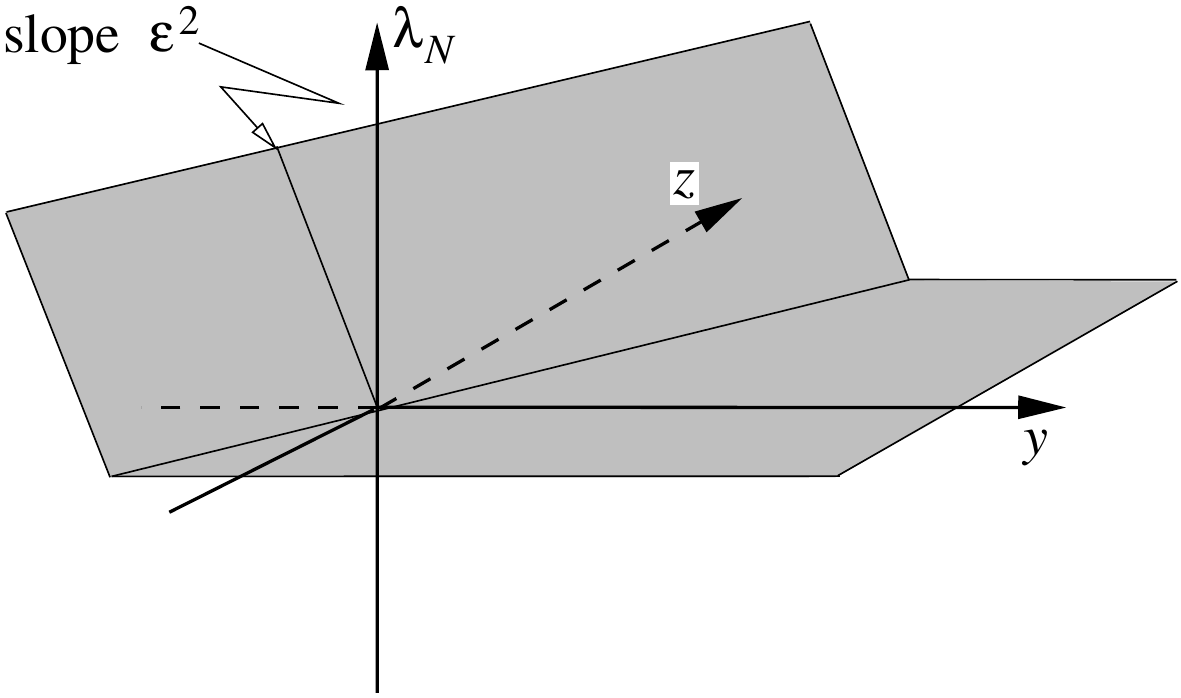}
\caption{(a) Schematic diagram of the compliant surface model. (b) 
The compliant normal force versus displacement relationship, which 
reduces to the usual rigid, unilateral contact law in the limit $\eps\to 0$.}
\label{fig:compliance}
\end{center}
\end{figure}

\begin{equation}
c \dot{z} = -k_1 z - \lambda_N
\label{eq:compliant1}
\end{equation}
\begin{equation}
\lambda_N = 
\left \{ \begin{array}{ll} 
k_2 (z-y)  & \mbox{if } y<z , \\
0 &  \mbox{otherwise} , 
\end{array} 
\right . 
\label{eq:compliant2}
\end{equation}
where $k_{1,2}$ are $\mathcal{O}(\eps^{-2})$ and $c=\mathcal{O}(\eps^{-1})$
for some small parameter $0<\eps\ll 1$.
For convenience in what follows we choose 
\beq
k_1=k_2=\frac{1}{\eps^2}, \qquad c= \frac{1}{\eps}.
\label{eq:compliant3}
\eeq

Note that this form of compliance, via the additional scalar 
deformation variable $z$, has an advantage over other forms of contact
regularisation reviewed in \cite{PainleveReview} because both
lift-off and touch-down are given by the same condition
$y=z$. Moreover, as we take the limit $\eps \to 0$, note that $z$
quickly relaxes to be equal to $y/2$ whenever $y<0$ and equal to zero
for $y>0$. Moreover, the level of deformation for a given force $\lambda_N$
tends to zero as $\eps \to 0$. 

It is also worth noting that this impact model is consistent (in the
limit of $\epsilon\rightarrow 0$) with rigid impact models based on
coefficients of restitution. For example, our model predicts an
ideally elastic impact (energetic coefficient of restitution $=1$ in
the sense of \cite{Stronge2000}) in any of the following limits:
$c\rightarrow 0$, $c\rightarrow\infty$, $k_2\rightarrow 0$, or
$k_1\rightarrow\infty$. It predicts ideally inelastic impact
(coefficient of restitution $ =0$) if $k_1\rightarrow 0$ and
simultaneously $k_2\rightarrow\infty$. Nevertheless fixed values of
$k_1$, $k_2$ and $c$ do not correspond to fixed values of the
coefficient of restitution in general.

Using  
\eqref{eq:compliant1}--\eqref{eq:compliant3} the compliant
version of the system is \eqref{eq:genTN} together with
\begin{align}
 \eps\dot{z}&=-z-\eps^2\lambda_N \label{eq:compl1}
 \\ 
   \lambda_N &= 
\left \{ \begin{array}{ll} 
\eps^{-2}( z-y )  & \mbox{if } y(\xi)<z \\
0 &  \mbox{otherwise} 
\end{array}\right .  \label{eq:compl2}.
\end{align}
and for positive slip \eqref{eq:genTN} becomes \eqref{eq:genx}.

\subsection{A simple motivating example}

Now, if the quantities $\alpha_1$, $\alpha_2$ and $\alpha_3$ are assumed to
be constant
in \eqref{eq:gensys2}--\eqref{eq:gensys5} then we get
\begin{align}
  \dot{p}&=\alpha_1^*,\label{eq:modelsys1}\\
  \dot{b}&=\alpha_2^*-\alpha_3^*\lambda_N,\label{eq:modelsys2}\\
  \dot{y}&=v,\label{eq:modelsys3}\\
  \dot{v}&=b+p\lambda_N,\label{eq:modelsys4}\\
 \eps\dot{z}&=-z-\eps^2\lambda_N, \label{eq:modelsys5}
\end{align}
where $\lambda_N$ is given by \eqref{eq:compl2}. 
Assuming $y<z$ in
\eqref{eq:compl2}, this set of equations admits a trivial solution
\begin{align}
  \bar{p}&=\alpha_1^*t,\label{eq:modelSysDistP}\\
  \bar{b}&=\frac{\alpha_2^*}{\alpha_1^*-\alpha_3^*}\alpha_1^*t,\label{eq:modelSysDistB}\\
  \bar{z}&=\frac{\bar{y}}{2}=-\eps^2\frac{\alpha_2^*}{\alpha_1^*-\alpha_3^*},\label{eq:modelSysDistZY}\\
  \bar{v}&=0\label{eq:modelSysDistV}
\end{align}
which, as we will see later, is a canard solution that passes through the 
$G$-spot in the limit $\eps \to 0$.

\subsection{A less degenerate example}

The question of whether IWC initiated at the G-spot terminates in stick, or 
in slip (as illustrated in Fig.~\ref{fig:qual}(b)) 
cannot be determined in the above model 
because there is no variation of the tangential velocity. Thus 
the simple system \eqref{eq:modelsys1}--\eqref{eq:modelsys5}
can never undergo a transition to stick. 
To allow
investigation of such a question, we need to add tangential
degrees of freedom to the model via introduction of variables $x$ and 
\begin{align}
\dot x = u \label{eq:xdot}
\end{align} 
representing
the tangential position and velocity of the contact point. We also have 
to include the non-smooth Coulomb friction law. The
contact force in stick or slip can be expressed as the sum 
of a forward slipping and a backward slipping
contact forces. Let the magnitude of the normal components 
of these two forces be given by $\lambda^+$ and $\lambda^-$, respectively. 
Specifically we can write
\beq
\lambda^+ =c\lambda_N, \qquad \lambda^- =(1-c)\lambda_N \label{eq:lambda+}
\eeq
where $c = 1$ for positive slip, $c = 0$ for negative slip, 
and for stick $c$ takes an intermediate value 
that shall be determined shortly.

The contact forces corresponding to positive and negative 
slip have different effects on the dynamics, therefore the 
terms $\alpha_3^*\lambda_N$ and $p\lambda_N$ of \eqref{eq:modelsys2} 
and \eqref{eq:modelsys4} must be replaced by general functions of 
$\lambda^+$ and $\lambda^-$. For simplicity in what follows we let the 
contact-dependent part of $\dot{b}$ be $\alpha_3^* \lambda^+$ and 
contact-dependent part of $\dot{v}$ be 
$p\lambda^++p^-\lambda^-$, where $p^-$ is a scalar. 

A similar distinction is made in the dynamics of the
new variable $u$, which is modelled by the equation
\begin{equation}
\dot u = a+k^+\lambda^+ + k^-\lambda^-\label{eq:udot}
\end{equation}
where $a$, $k^+$, $k^-$ are also scalars. 
The condition $\dot u=0$ for stick now allows us to determine the missing
value 
\begin{equation}
c=(k^-+\lambda_N^{-1}a)/(k^--k^+).
\label{eq:cdot}
\end{equation}

In addition to the necessary extensions outlined above, we 
also introduce a parametric state dependence of $\alpha_1$ in the form of 
$\alpha_1(\xi)=\alpha_1^*+\chi b$, for some scalar $\chi$
which allows two-way coupling between normal and tangential dynamics. 
The resulting 
extended example system can now be written in the form 
\eqref{eq:compl2}, \eqref{eq:xdot}, \eqref{eq:udot} and
\begin{align}
 \dot{p}&=\alpha_1^*+\chi b, \label{eq:newex1}\\
  \dot{b}&=\alpha_2^*-c \alpha_3^*\lambda_N
\label{eq:newex2}\\
  \dot{y}&=v, \label{eq:newex3}\\
  \dot{v}&=b+pc\lambda_N+p^-(1-c)\lambda_N, \label{eq:newex4}\\
 \eps\dot{z}&=-z-\eps^2\lambda_N, \label{eq:newex5}
\end{align}
in which $\chi$, $\alpha^*_{1-3}$, $k^{\pm}$ and $a$ are fixed constants
and $c$ is given by \eqref{eq:cdot} for stick and is equal to 0 or 1 for
negative and postive slip respectively.

If $p\approx 0$ then the negative \pain parameter must be
positive \cite{PainleveReview}, hence we choose $p^- = 1$. Finally, 
positive (respectively, negative) slipping contact forces
typically accelerate the contact point in the negative (positive) tangential direction, which motivates the choice  
\beq
p^- =1, \quad k^+ = -1, \quad k^- = 2 \quad \mbox{and} \quad 
a=0. 
\label{eq:newex6}
\eeq
where the choice $a=0$ is simply made for convenience.

Note that this extended example system 
\eqref{eq:compl2}, \eqref{eq:xdot}-\eqref{eq:udot} and 
\eqref{eq:newex1}--\eqref{eq:newex6} was not explicitly derived from a Lagrangian system, nevertheless it can be written in the form \eqref{eq:genTN} by taking
\begin{equation*}
  F(\xi)=\begin{pmatrix}\xi_2\\0\\ \xi_4\\ \xi_6\\ \alpha_1^*+\chi \xi_6\\ \alpha_2^*
  \end{pmatrix}, \quad 
  G_T(\xi)=\begin{pmatrix}0\\\frac{3}{2(1-\xi_5)}\\ 0\\ \frac{1}{2}\\ 0\\ \frac{\alpha_3^*}{2(1-\xi_5)}
  \end{pmatrix}, \quad 
  G_N(\xi)=\begin{pmatrix}0\\ \frac{1}{2}\\ 0\\ \frac{1+\xi_5}{2}\\ 0\\ -\frac{\alpha_3^*}{2}
  \end{pmatrix},\quad \mu=1-\xi_5, \quad x=\xi_1, \quad y=\xi_3.
\end{equation*}
and it satisfies the relations \eqref{eq:lagrangian1}, 
\eqref{eq:lagrangian2}, which reflect the Lagrangian character 
of general systems. It then follows that
\begin{equation*}
 u=\xi_2, \quad  v=\xi_4, \quad p=\xi_5, \quad b=\xi_6, \quad \alpha_1= \alpha_1^*+\chi \xi_6, \quad \alpha_2= \alpha_2^*, \quad \alpha_3= \alpha_3^*.
\end{equation*}
Note that the parameter $\chi$ can be effectively thought of as a homotopy parameter that allows us to pass from a simple case ($\chi=0$) in which there
is a trivial solution (similar to 
\eqref{eq:modelSysDistP}--\eqref{eq:modelSysDistV}) that passes through the 
$G$-spot to a more complicated case ($\chi=1$) in which there is no such trivial
solution. 

\section{Numerical results for motivating example}
\label{sec:3}

We consider first the simplified version of the motivating example
\eqref{eq:modelsys1}--\eqref{eq:modelsys5}.
We want to understand what happens to initial conditions 
that are small perturbations from
the trivial solution \eqref{eq:modelSysDistP}--\eqref{eq:modelSysDistV}. 

\subsection{A dichotomy between lift-off and impact}

\begin{figure}
\begin{center}
\includegraphics[width=\textwidth]{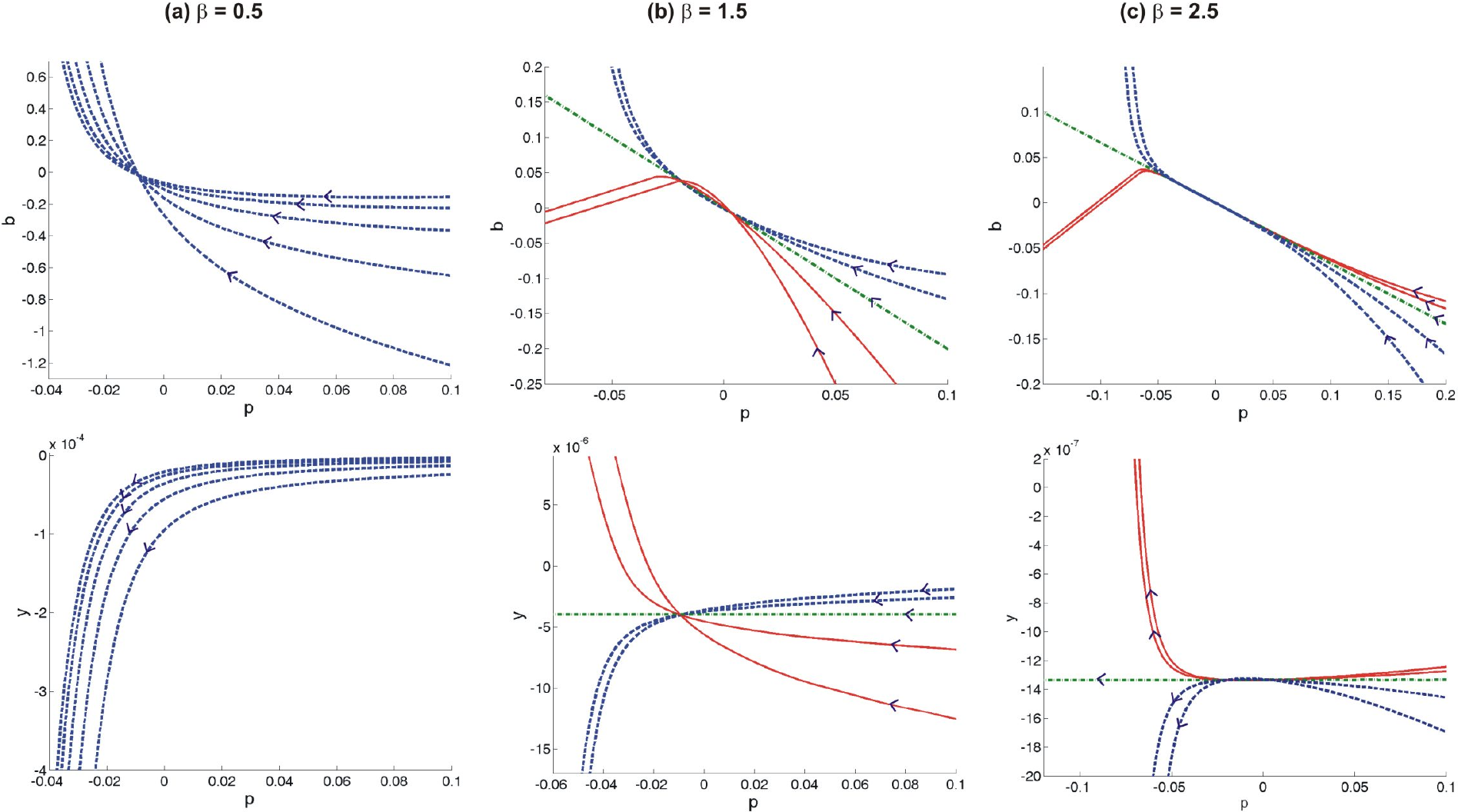}
\caption{Numerical simulation of \eqref{eq:modelsys1}--\eqref{eq:modelsys4}
for $\eps=10^{-3}$ with $\alpha_1^*=\alpha_2^*=-1$ and $\alpha_3^*=-\beta$,
where; (a) $\beta= 0.5$, 
(b)  $\beta=1.5$, or (c) $\beta=2.5$. 
Initial conditions in each case are $p(0)=0.5$; 
$b(0)=-| \nu \cdot p(0)\alpha_2^*/(\alpha_3^*-\alpha_1^*)|$ where $\nu=0.25$, $0.5$, 1, 2, 4 $y(0)=2\eps^2b(0)/p(0)$, $z(0)=\eps^2b(0)/p(0)$, $v(0)=0$.
In each plot, a dot-dashed (green online) line depicts the trivial solution, 
solid (red online) curves represent trajectories that lift off,
whereas dashed (blue online) curves represent trajectories that take an IWC. 
}
\label{fig:varsim1}
\end{center}
\end{figure}

First of all, note that the internal dynamics of the compliant contact
model creates damped oscillations.  Clearly, this is an artefact of
our contact model and not important to our discussion. Note though an
important feature of these oscillations is that their frequency
diverges to $\infty$ in the limit of $\eps\rightarrow 0$. This lack of
smoothness allows us to separate this component of the dynamics in any
subsequent analysis. In our preliminary simulations we minimise
transient oscillations by choosing initial conditions satisfying
$z=y/2$, $v=0$ and $y=2b\eps^2/p$ and by choosing an initial value
$p=0.5$, which is sufficiently distant from the G-spot to allow for
the $z$ dynamics to relax.

The results of three
simulations for different values of
the scalar parameter
\begin{equation}
  \beta=\frac{\alpha_3^*}{\alpha_1^*}
\label{eq:beta}
\end{equation}
are presented in Fig. \ref{fig:varsim1}. 
Panel (a) of the figure
corresponds to Case I where $\alpha_1^*< \alpha_3^* < 0$. Here we see
that for all initial conditions that become attracted to the G-spot,
$y$ diverges to $-\infty$, which indicates the onset of an IWC. 
The trivial solution in this case is 
unphysical ($bp > 0$, implying $\lambda_N<0$) and is not shown. 
We found the results in
Case II to be similar, that is, all initial conditions that pass the
G-spot take an impact. 

Panels (b) and (c) of Fig.~\ref{fig:varsim1} 
illustrate two different
examples of Case III where $\beta$ is 1.5 and 2.5,
respectively. Here we see that there is a dichotomy, in that there
are some trajectories that immediately lift off 
(which can be seen
because $y$ increases rapidly), whereas other trajectories 
take an IWC.
The trivial solution appears to form a
separatrix between these two behaviours. Interestingly, the set of
initial conditions that impact or lift off are swapped, between the two 
examples shown. That is, initial conditions with lower initial values
of $b(0)$ are the ones that take an impact in panel (b) whereas it is
those with the higher $b(0)$ that take an impact in panel (c).

\subsection{Smoothness in the limit of $\epsilon\rightarrow 0$}

The trajectories presented above not only differ in their asymptotic
behaviour for large times, but also in their degree of smoothness as a function
of time in the limit $\epsilon\rightarrow 0$. To illustrate this property, we
have repeated the same simulations with $\epsilon=10^{-5}$. The
results are illustrated as plots of $y$ as a function of $p$ in
Fig. \ref{fig:smoothness} for all three values of $\beta$. Note that $p$
scales linearly with time.  The trivial
separatrix solutions appear for $\beta=1.5$ and 2.5 as a straight line,
which is by definition, infinitely smooth. 
At the same time, all other trajectories show
some kind of divergence as $p \to 0$. For $\beta =0.5$, $y$ appears to
to diverge to infinity in the limit of small $p$. In contrast,
for $\beta=1.5$, the first derivative of $y$ appears to
diverge. For $\beta=2.5$, a more detailed analysis (not shown) shows
divergence of the second derivative of $y$. As we shall see shortly, 
systematic variation of $\epsilon$ also confirms these observations.

Now, for the simple model system, we have a trivial solution that forms
the separatrix. For more general systems (as for example
\eqref{eq:newex1}--\eqref{eq:newex6} with $\chi=1$) we shall demonstrate in 
Sec.~\ref{sec:5} below that there nevertheless exists a separatrix trajectory
in Case III, which
preserves a higher
degree of smoothness in the limit $\eps\rightarrow 0$ than any other
trajectory. Specifically, the  method of construction will be used to develop an
expansion for the trajectory that is at least $C^\infty$ in $t$ and $\eps$. 
This property enables us to disregard 
all other trajectories (either with or without oscillatory
components) that are less smooth in the limit $\eps \to 0$.

\begin{figure}
\begin{center}
\includegraphics[width=0.5\textwidth]{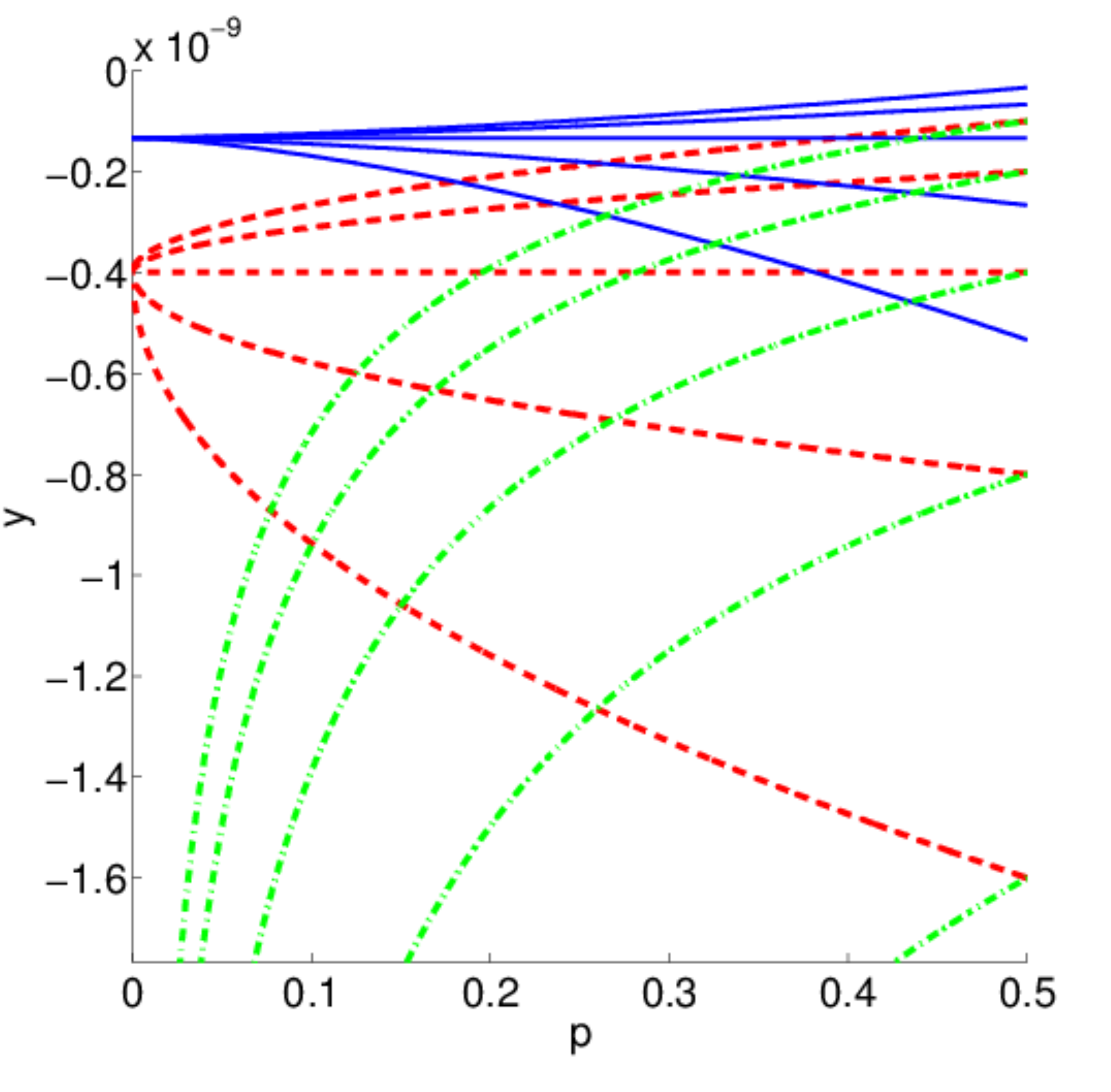}
\caption{Numerical simulation of \eqref{eq:modelsys1}--\eqref{eq:modelsys4}
for $\eps=10^{-5}$ with:  dot-dashed line (green online) $\beta=0.5$; 
dashed line (red online) $\beta=1.5$; and solid line (blue online) $\beta=2.5$.
Other values are the same as in previous simulations.}
\label{fig:smoothness}
\end{center}
\end{figure}

\subsection{Asymptotic behaviour for $\epsilon\rightarrow 0$}
\label{sec:aymptoticsim}

Our analysis of the dynamics near the G-spot will make use of a
carefully chosen {\em inner scaling} of the variables. To motivate the
particular scaling chosen in Sec.~\ref{sec:5}, we now present
the numerically observed dependence on
$\epsilon$ of the dynamics of the model
system. We will use letters with a hat ( $\hat{}$ ) for the deviation
of variables $p$, $b$, $y$, $v$, $z$ from their values along the
trivial solution, for example.
$$
b(t,\epsilon)=\bar{b}(t,\epsilon)+\hat b(t,\epsilon).
$$
To learn how $\hat b$, $\hat y$ and $\hat v$ scale as
$\epsilon\rightarrow 0$, we have recorded their values at the time of
passing the G-spot ($p=0$) in a series of simulations where
$\epsilon$ was varied systematically. Three values of $\beta$
 were
considered and the initial conditions used were the same as in the caption
of Fig.~\ref{fig:varsim1} with $\nu=2$. The results for $\hat b$ are
depicted in Fig.~\ref{fig:sim-bscaling}. We found a nearly perfect
power-law relationship $\hat b\approx \epsilon^\gamma$ for constant times
determined by $p=0$, at least for sufficiently small 
$\epsilon$,  where the exponent $\gamma$ was determined by linear regression,
see Table \ref{tab:exponents}(a).
Similar results were obtained for $\hat v$ and $\hat y$. 

\begin{figure}
\begin{center}
\includegraphics[width=0.5\textwidth]{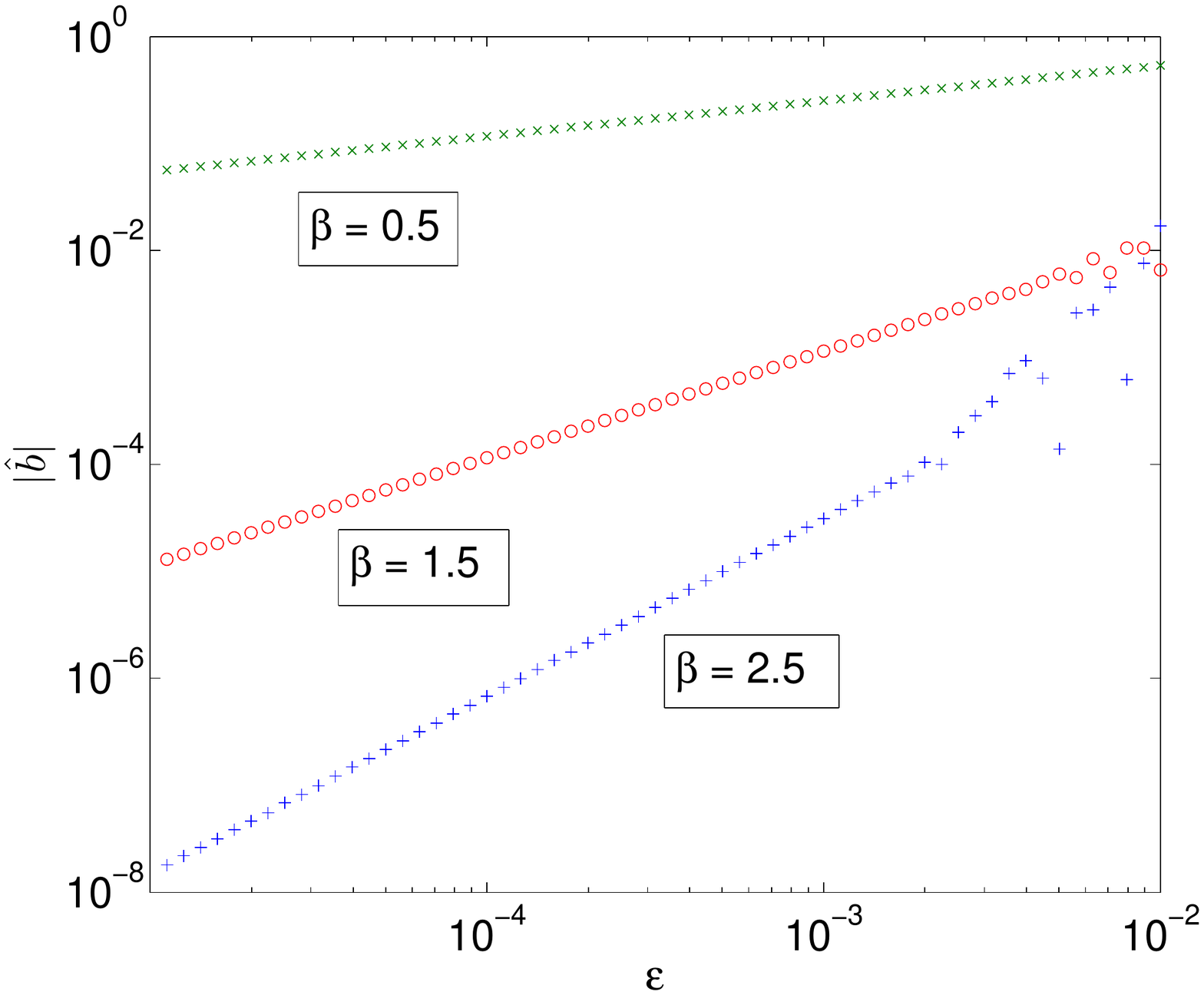}
\caption{Logarithmic plot of $\hat b$ at the time when $p=0$ as $\eps$ varies,
for a perturbed initial condition close to the $G$-spot. See text for details.
} 
\label{fig:sim-bscaling}
\end{center}
\end{figure}

In a similar manner, we have also measured the time
difference between passing the G-spot ($p=0$) and crossing
$\hat b=0$, see the last row of the table. Clearly the exponent of the
time-difference is close to 2/3 whereas other exponents appear to
depend on $\beta$ linearly. 
The general asymptotic theory in Secs.~\ref{sec:4} and \ref{sec:5} below 
predict that these exponents should, in the limit $\eps \to 0$, take the 
values $2\beta/3$ ($\hat b$), $(2\beta+2)/3$ ($\hat
v$) and $(2\beta+4)/3$ ($\hat y$).
Table \ref{tab:exponents}(b) gives the theoretical
values according to these formulae. We see that there is excellent agreement
with the numerical findings. 

\begin{table}
\begin{center}
\begin{tabular}{c|ccc|c|ccc|}
\cline{2-8}
& \multicolumn{3}{|c|}{ (a) {\bf Measured}} &
\multicolumn{4}{|c|}{ (b) {\bf Theoretical}}  \\
\cline{2-8}

&  $\beta=0.5$ & $\beta=1.5$ & $\beta=2.5$ & & 
$\beta=0.5$ & $\beta=1.5$ & $\beta=2.5$ \\  
\hline
\multicolumn{1}{|c|}{exponent for $\hat b$} &0.3326&0.9982&1.6635 & $2\beta/3$  & 0.3333 & 1.0  & 1.6667\\ 
\multicolumn{1}{|c|}{exponent for $\hat v$} &0.9983&1.6652&2.3396 &  $(2\beta+2)/3$ &  1.0 & 1.6667 & 2.3333\\
\multicolumn{1}{|c|}{exponent for $\hat y$} & 1.6538&2.3384&2.9808 &  $(2\beta+4)/3$   & 1.6667 & 2.3333 & 3.000\\
\multicolumn{1}{|c|}{exponent for $t$} & 0.6650 & 0.6656 &0.6659 &    $2/3$  &  0.6667 & 0.6667 &0.6667  \\
\hline 
\end{tabular}
\caption{(a) Numerically measured scaling exponents $\gamma$ such that the
named quantity in the first column scales like $\eps^\gamma$  
as $\epsilon\rightarrow 0$. (b) Theoretical values of these exponents
according to the asymptotic theory of Secs.~\ref{sec:4} and \ref{sec:5}.}
\label{tab:exponents}
\end{center}
\end{table}

\subsection{Possible dynamics beyond the G-spot}

The above dynamics simply illustrate the scaling of trajectories and
whether they lift off or take an IWC. What
happens after these two possible events is also interesting in its own
right.  Firstly, after lift-off we have $\lambda_N = 0$, and thus
$\dot b=\alpha_2^*$. In Case II, $\dot b > 0$, which means that $b$,
$y$ and $v$ will increase, and lift-off will persist for some
time. Nevertheless in Cases I and III, $\dot b < 0$, which eventually
cause $v$ and $y$ to decrease as well. Hence, lift-off will --- at
least in the limit of $\eps\rightarrow 0$ --- always terminate shorty
after passing the G-spot, and an impact with very low pre-impact
normal velocity (i.e. a quasi-IWC) will occur.

In order to examine what happens once an IWC is initiated, we have
to consider the extended version of the example system \eqref{eq:compl2}, \eqref{eq:xdot}-\eqref{eq:udot} and \eqref{eq:newex1}--\eqref{eq:newex5},
which includes tangential dynamics and possible transitions from slip
to stick. 

Two simulations are presented in Fig.~\ref{fig:varsim2}. In each one,
we use the same parameter values and some of the 
initial conditions used in 
Fig.~\ref{fig:varsim1}(a), but compute for a longer timespan. 
The initial value of $u$ is $u(0) = 70$ 
whereas the initial value of $x(0)$ is arbitrary,
since $x$ is a cyclic coordinate.  Furthermore we use $\chi = 0$ in
the first simulation and $\chi = 1$ in the second.  
In the first case (continuous curve, red online), 
we observe that the near-tangential impact continues all
the way until the contact sticks ($\dot u = 0$). However in the second
(dashed curve, blue online), the large contact force initiates a rapid
increase of $b$ (due to $\alpha_3^*<0$). For large enough $u(0)$, the
variables $p$, $v$ and $y$ all begin to increase before the contact
sticks. In this case the impact will terminate and a lift-off occur before
we reach all the way to $u=0$.

\begin{figure}
\begin{center}
\includegraphics[width=0.8\textwidth]{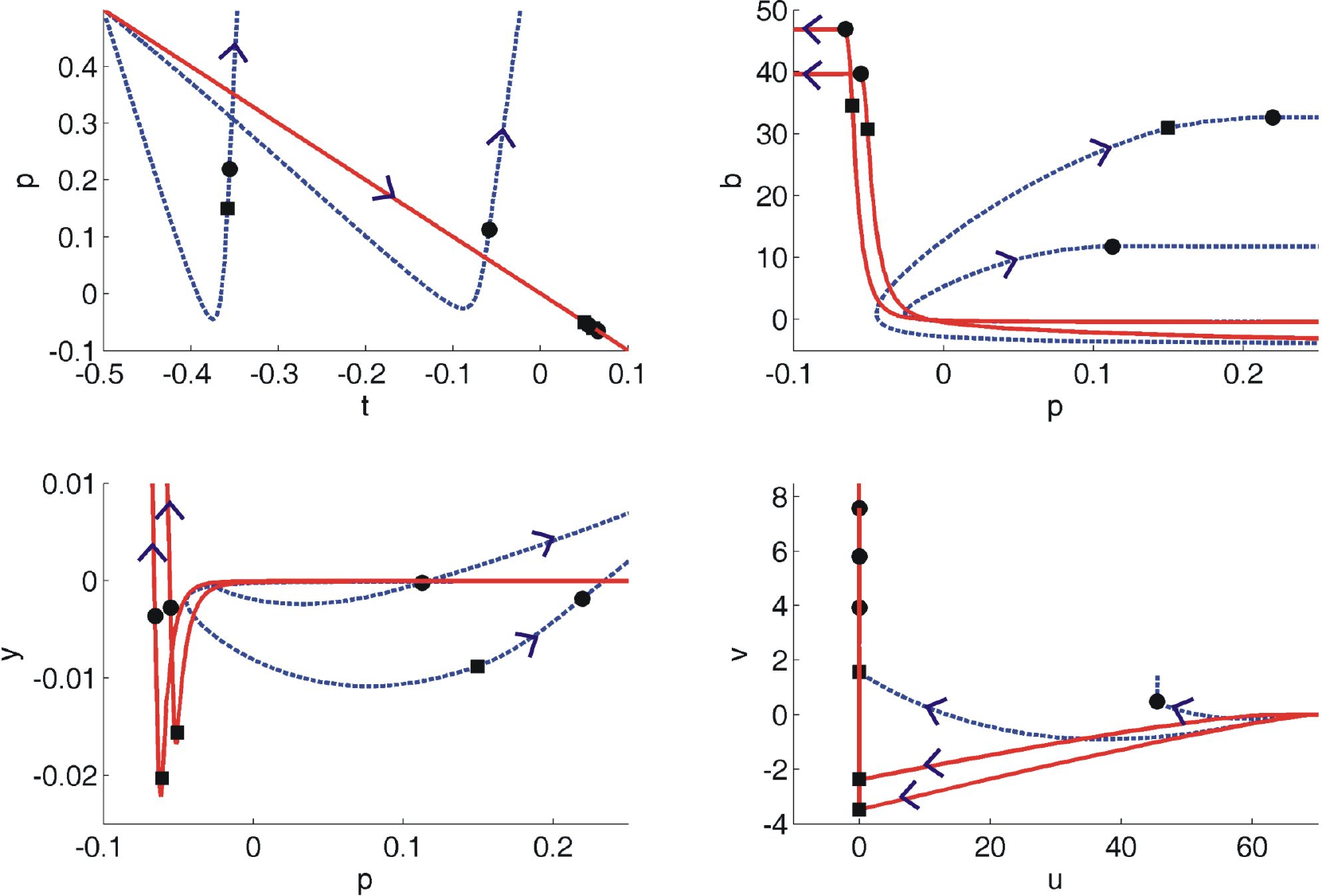}
\caption{ Two-dimensional projections of simulation results of the
 extended model system with $u(0)=70$ and $\chi=0$ (solid curve, red
 online) or $\chi=1$ (dashed curve, blue online). Other parameter
 values are as in Fig.~\ref{fig:varsim1}(a) except that results
for two different $\nu$-values are shown,  namely $\nu=0.25$ and
 $4$. Squares and circles denote slip-stick transitions and liftoff,
 respectively. Notice that liftoff occurs without slip-stick
 transition in one out of four cases.  }
\label{fig:varsim2}
\end{center}
\end{figure}

\section{Asymptotic analysis of the motivating example}
\label{sec:4}

Before presenting general analysis, it is useful to explore the major ideas 
using the simplified version of the model system 
\eqref{eq:modelsys1}--\eqref{eq:modelsys5} for
which the details are eased because of the 
existence of a trivial smoothest solution
\eqref{eq:modelSysDistP}--\eqref{eq:modelSysDistV}.
We assume throughout that the system is in contact, so that
\begin{equation*}
  \eps^2\lambda_N = z-y,
\end{equation*}
and we suppose that $\alpha_1^*<0$, $\alpha_3^*<0$.

If we fix the origin of time so $p(0)=0$, then
$p(t)=\alpha_1^*t$. Inserting this into the other equations, 
we can eliminate $y$, $v$, and $\lambda_N$ by differentiating 
the $\dot{b}$ equation with respect to time three times and eliminating 
$z$ and its derivatives via 
$$
\alpha_3^*(z-y) = \eps^2\left (\alpha_2 - \frac{db}{dt} \right), \qquad
\alpha_3^*(\dot{z}-\dot{y}) = -\eps^2\frac{d^2b}{dt^2}, \qquad 
\alpha_3^*(\ddot{z}-\ddot{y}) = -\eps^2\frac{d^3b}{dt^3}.
$$
We obtain
\begin{equation*}
  \alpha_3^*b+\alpha_1^*t\left(\alpha_2^*-\frac{db}{dt}\right)+\epsilon\left[\alpha_3^*\frac{db}{dt}+\alpha_1^*\left(\alpha_2^*-\frac{db}{dt}\right)-\alpha_1^*t\frac{d^2b}{dt^2}\right]-2\epsilon^2\frac{d^3b}{dt^3}-\epsilon^3\frac{d^4b}{dt^4}=0,\label{eq:modelSysTEliminated}
\end{equation*}
which is a more convenient 3rd-order single equation for $b(t)$. 
Note that the other variables 
can be recovered via
\begin{align}
  \alpha_3^*y&=\epsilon^2\left[\alpha_3^*b+\alpha_1^*t\left(\alpha_2^*-\frac{db}{dt}\right)-2\left(\alpha_2^*-\frac{db}{dt}\right)\right]-\epsilon^3\frac{d^2b}{dt^2}-\epsilon^4\frac{d^3b}{dt^3},\label{eq:modelSysRecoverY}\\
  \alpha_3^*v&=-\epsilon\left[\alpha_3^*b+\alpha_1^*t\left(\alpha_2^*-\frac{db}{dt}\right)\right]+2\epsilon^2\frac{d^2b}{dt^2}+\epsilon^3\frac{d^3b}{dt^3},\label{eq:modelSysRecoverV}\\
  \alpha_3^*z&=\epsilon^2\left[\alpha_3^*b+\alpha_1^*t\left(\alpha_2^*-\frac{db}{dt}\right)-\left(\alpha_2^*-\frac{db}{dt}\right)\right]-\epsilon^3\frac{d^2b}{dt^2}-\epsilon^4\frac{d^3b}{dt^3}, \label{eq:modelSysRecoverZ}\\
  \alpha_3^*\lambda_N&=\alpha_2^*-\frac{db}{dt} \label{eq:modelSysRecoverLN}.
\end{align}

We now want to find a time rescaling to de-singularise the
$G$-spot. One possibility would be to rescale time using the value of
$p$, \eqref{eq:time_rescale}, leading to
\eqref{eq:Gspot1}--\eqref{eq:Gspot2} as in \cite{Genot1999}.
Unfortunately, such a rescaling is too brutal to obtain information on
what happens beyond the G-spot, not least because the system can only
be defined for $p>0$.  Instead, we shall seek a scaling in terms of
the parameter $\eps$ of the contact regularised system. In so doing we
will get an {\em outer} dynamical system, which will take the form of
a fast-slow system \cite{Kuhn}.  Then we introduce a new {\em inner}
timescale which is $\mathcal{O}(\eps^{2/3})$. This gives the ability
to find a distingished limit in which the singularity associated with
the $G$-spot is balanced by the contact dynamics.

\subsection{A distinguished trajectory}

We shall start by considering the explicit trivial solution 
\eqref{eq:modelSysDistP}--\eqref{eq:modelSysDistV}. Note that this solution
is smooth in both the variables $t$
and $\epsilon$. If the parameter $\beta$ (see \ref{eq:beta})
is not an integer, we will now show that it is the only solution that is
smooth in $t$ and $\epsilon$, and thus we designate it as the 
{\em distinguished trajectory}.

Such a smooth trajectory must have an expansion in $\epsilon$
\begin{equation*}
  \bar{b}(t,\epsilon)=\sum_nb_n(t)\epsilon^n
\end{equation*}
such that each $b_n(t)$ is a smooth function. Inserting this into
\eqref{eq:modelSysTEliminated}, we find to order $\epsilon^0$ that
\begin{equation*}
  \alpha_3^*b_0+\alpha_1^*t\left(\alpha_2^*-\frac{db_0}{dt}\right)=0,
\end{equation*}
which has general solution
\begin{equation*}
  b_0(t)=\frac{\alpha_2^*}{1-\beta}t+
  \begin{cases}
    C_1(-t)^\beta & \text{if $t\le0$}\\
    C_2(t)^\beta & \text{if $t\ge0$}
  \end{cases}.
\end{equation*}
Since $\beta$ is not an integer, $b_0$ is not smooth unless
$C_1=C_2=0$.
Inserting this solution into the order $\epsilon^1$ equation we get
\begin{equation*}
  \alpha_3^*b_1-\alpha_1^*t\frac{db_1}{dt}=0,
\end{equation*}
and again smoothness forces $b_1(t)=0$. Proceeding similarly, at 
$\mathcal{O}(\epsilon^n)$, we get 
\begin{equation*}
  \alpha_3^*b_n-\alpha_1^*t\frac{db_n}{dt}=0
\end{equation*}
and thus we need to choose $b_n(t)=0$ for all $n>0$.

Hence the requirement of smoothness leads uniquely to
\begin{equation*}
\bar{b}(t,\epsilon)=\frac{\alpha_2^*}{1-\beta}t,
\end{equation*} 
from which we can recover the rest of the trivial solution
\eq{eq:modelSysDistP}--\eq{eq:modelSysDistV}, for the 
variables $y$, $v$ and $z$ using
\eqref{eq:modelSysRecoverY}--\eqref{eq:modelSysRecoverZ}.

\subsection{Deviations from the distinguished trajectory: outer
  scaling}

We now wish to consider trajectories whose initial 
conditions near the G-spot are small perturbations from the distinguished
trajectory. Recall that $\hat{b}$ denotes deviations from $\bar{b}$. 
Then, $\hat{b}$ satisfies
\begin{equation}
  \alpha_3^*\hat{b}-\alpha_1^*t\frac{d\hat{b}}{dt}+\epsilon\left[ \left(\alpha_3^*-\alpha_1^*\right)\frac{d\hat{b}}{dt}-\alpha_1^*t\frac{d^2\hat{b}}{dt^2}\right]-2\epsilon^2\frac{d^3\hat{b}}{dt^3}-\epsilon^3\frac{d^4\hat{b}}{dt^4}=0,
\label{eq:modelSysDev}
\end{equation}
while the deviations of the other variables from the distinguished
trajectory can be recovered using
\begin{align}
  \alpha_3^*\hat{y}&=\epsilon^2\left[\alpha_3^*\hat{b}-\alpha_1^*t\frac{d\hat{b}}{dt}+2\frac{d\hat{b}}{dt}\right]-\epsilon^3\frac{d^2\hat{b}}{dt^2}-\epsilon^4\frac{d^3\hat{b}}{dt^3}\label{eq:modelSysDevRecoverY}\\
  \alpha_3^*\hat{v}&=-\epsilon\left[\alpha_3^*\hat{b}-\alpha_1^*t\frac{d\hat{b}}{dt}\right]+2\epsilon^2\frac{d^2\hat{b}}{dt^2}+\epsilon^3\frac{d^3\hat{b}}{dt^3}\label{eq:modelSysDevRecoverV}\\
  \alpha_3^*\hat{z}&=\epsilon^2\left[\alpha_3^*\hat{b}-\alpha_1^*t\frac{d\hat{b}}{dt}+\frac{d\hat{b}}{dt}\right]-\epsilon^3\frac{d^2\hat{b}}{dt^2}-\epsilon^4\frac{d^3\hat{b}}{dt^3}\label{eq:modelSysDevRecoverZ}\\
  \alpha_3^*\hat{\lambda}_N&=-\frac{d\hat{b}}{dt}\label{eq:modelSysDevRecoverLN}.
\end{align}

Assuming $t$ is not close to zero, we can identify two timescales
in \eqref{eq:modelSysDev}; a slow timescale of order $\mathcal{O}(1)$
and a fast timescale of order $\mathcal{O}(\epsilon)$.

\paragraph*{The fast system.} Introducing a fast timescale $t_f$ via 
$dt = \eps dt_f$, and reintroducing $p=\alpha_1^* t$,
we find that equation \eqref{eq:modelSysDev} becomes
\begin{equation*}
  -\left[p\left(\frac{d\hat{b}}{dt_f}+\frac{d^2\hat{b}}{dt_f^2}
\right)+2\frac{d^3\hat{b}}{dt_f^3}+\frac{d^4\hat{b}}{dt_f^4}\right]+\epsilon
\left[\alpha_3^*\hat{b}+\left(\alpha_3^*-\alpha_1^*\right)\frac{d\hat{b}}{dt_f}
\right]=0.
\end{equation*}
Setting $\eps=0$, thus treating $p$ as a constant, and looking for
exponential solutions to the resulting linear constant coefficient
equation, we get the characteristic polynomial 
$$
-\lambda\left[\lambda^3+2\lambda^2+p\lambda+p\right]=0
$$ 
where the
first factor gives a zero root corresponding to the slow time scale,
whereas the non-trivial second factor corresponds to the dynamics of the fast
system. For this second factor, if $p>0$, the Routh-Hurwith criterion
implies that all roots have negative real parts, hence the fast
subsystem is stable. Hence, for $p>0$ trajectories are
attracted to a codimension-three manifold, representing the slow
dynamics.

In contrast, if $p<0$, then the discriminant of the second factor 
$\Delta=-p(4(p)^2-13p+32)$ is always positive, which means that 
the fast system has three real eigenvalues. Note further that the sum
of the eigenvalues is $-2$ whereas their product is $-p$, which implies
that precisely one eigenvalue out of the the three is positive for $p<0$. 
Hence the slow dynamics for $p<0$ is normally hyperbolic with a two-dimensional
stable manifold and one-dimensional unstable manifold.  

\paragraph{The slow system} is obtained from 
\eqref{eq:modelSysDev} by letting $\epsilon\rightarrow0$ so that we obtain  
\begin{equation*}
  \alpha_3^*\hat{b}-\alpha_1^*t\frac{d\hat{b}}{dt}=0,
\end{equation*}
with solution
\begin{equation}
  \hat{b}(t)=
  \begin{cases}
    C_1(-t)^\beta & \text{if $t\le0$}\\
    C_2(t)^\beta & \text{if $t\ge0$}\\
  \end{cases}
\label{eq:example_outer}
\end{equation}
and
\begin{equation*}
  \frac{\hat{y}}{\epsilon^2}=2\frac{\hat{b}}{\alpha_1^*t},\quad
  \frac{\hat{v}}{\epsilon}=0,\quad
  \frac{\hat{z}}{\epsilon^2}=\frac{\hat{b}}{\alpha_1^*t},\quad
  \hat{\lambda}_N=-\frac{\hat{b}}{\alpha_1^*t}, \quad
\beta=\frac{\alpha_3^*}{\alpha_1^*}.
\end{equation*}

%

\subsection{Inner scaling}

When $t$ is close to zero, we can introduce a rescaled 
time variable via
$$
t= \delta^2 s, \qquad \mbox{where} \quad  \delta=\epsilon^{1/3},
$$ 
in accordance with the observations in Table \ref{tab:exponents}(a). Under
such a rescaling, \eqref{eq:modelSysDev} becomes
\begin{equation}
  \alpha_3^*\hat{b}-\alpha_1^*s\frac{d\hat{b}}{ds}-2\frac{d^3\hat{b}}{ds^3}+\delta\left[\left(\alpha_3^*-\alpha_1^*\right)\frac{d\hat{b}}{ds}-\alpha_1^*s\frac{d^2\hat{b}}{ds^2}-\frac{d^4\hat{b}}{ds^4}\right]=0.\label{eq:modelSysDevInner}
\end{equation}
Again we can identify two timescales: a slow timescale of order
$\mathcal{O}(\epsilon^{2/3})$ ($\mathcal{O}(1)$ in $s$), and a fast
timescale of order $\mathcal{O}(\epsilon)$
($\mathcal{O}(\epsilon^{1/3})$ in $s$).

\paragraph*{The fast time scale} is defined by defining a new time
variable $s_f$ so that  $ds=\delta ds_f$. From this, \eqref{eq:modelSysDevInner}
becomes
\begin{equation*}
  -2\frac{d^3\hat{b}}{ds_f^3}-\frac{d^4\hat{b}}{ds_f^4}+\delta^2\alpha_1^*s\left(\frac{d\hat{b}}{ds_f}+\frac{d^2\hat{b}}{ds_f^2}\right)+\delta^3\left[\alpha_3^*\hat{b}+\left(\alpha_3^*-\alpha_1^*\right)\frac{d\hat{b}}{ds_f}\right]=0.
\end{equation*}
Letting $\epsilon\rightarrow0$ we have a characteristic equation
$-\lambda^3\left[\lambda+2\right]=0$. There are three zero
roots corresponding to the slow time scale, and one real root
$\lambda=-2$. We conclude that the fast system is stable. 

\paragraph*{The slow timescale} is obtained by letting
$\delta\rightarrow0$ in \eqref{eq:modelSysDevInner} from which we obtain
\begin{equation}
  \alpha_3^*\hat{b}-\alpha_1^*s\frac{d\hat{b}}{ds}-2\frac{d^3\hat{b}}{ds^3}=0.\label{eqn:de_inner}
\end{equation}
This equation has solutions that can be expressed in terms of
generalised hypergeometric functions. Specifically, upon rescaling 
$$
\tau=\kappa s, \qquad \mbox{with} \quad \kappa =(-\alpha_1^*/2)^{1/3} >0
$$ 
and setting $b(s) = \theta(\tau)$ we get
\beq
  \frac{d^3}{d \tau^3}\theta-\tau\frac{d}{ d\tau} \theta+\beta\theta=0,
\label{eq:hypergeomeqn}
\eeq
for
$$
\beta=\alpha_3^*/\alpha_1^*.
$$
The equation \eqref{eq:hypergeomeqn} has a 
solution that can be expressed in terms of generalised 
hypergeometric functions, whose asymptotic properties are summarised
in Appendix \ref{A:1}. From that solution, we can recover
all other variables from $\hat{b}$ via
\begin{equation}
  \delta^{2}\alpha_3^*\frac{\hat{y}}{\epsilon^2}=2\frac{d\hat{b}}{ds},\quad
  \delta^{2}\alpha_3^*\frac{\hat{v}}{\epsilon}=2\frac{d^2\hat{b}}{ds^2},\quad
  \delta^{2}\alpha_3^*\frac{\hat{z}}{\epsilon^2}=\frac{d\hat{b}}{ds},\quad
  \delta^{2}\alpha_3^*\hat{\lambda}_N=-\frac{d\hat{b}}{ds}.\label{eq:inner_scaling}
\end{equation}

\subsection{Matching the inner and outer solutions}
\label{subsec:4.4}
We know that the outer solution behaves like
\begin{equation}
\hat{b}(t) \sim C_1(-t)^\beta  \qquad \mbox{as} \quad 
t\rightarrow 0^-.
\label{eq:innerasymptote}
\end{equation}
and this must match
the behaviour of the inner solution $\hat{b}(s)$ as
$s\rightarrow-\infty$.

The asymptotics of the solution inner region
$s \to \pm \infty$ can be established from the asymptotic behavour of 
$\theta(\tau)$, given by 
\eqref{eq:hypergeomeqn} which is studied in Appendix \ref{A:1}. 
The general solution is a linear combination of a
function $h$ with power-law type behaviour and two rapidly oscillating
functions $e_r$ and $e_i$. To get the desired behaviour, the
coefficients of $e_r$ and $e_i$ must vanish, and there is a unique solution
$\Theta(\tau,\beta)$ that behaves like $(-\tau)^\beta$ as $\tau\rightarrow-\infty$. 

Matching with the small $t$ limit \eqref{eq:innerasymptote} with the
large negative $\tau$ limit, we 
find that the leading-order inner solution matches if we take it to be
\begin{equation}
  \hat{b}(s)=C_1\kappa^{-\beta}\eps^{2\beta/3}\Theta(\kappa
  s,\beta)\sim C_1\kappa^{-\beta}\eps^{2\beta/3}(-\kappa s)^\beta=
  C_1(-\epsilon^{2/3} s)= C_1(-t)^\beta, \label{eq:bhat_inner}
\end{equation}
where again $\kappa = (-\alpha_1^*/2)^{\frac{1}{3}}$.

Combining \eqref{eq:bhat_inner} with \eqref{eq:inner_scaling}, 
we find the scaling for all
the perturbation variables of the inner solution
\begin{eqnarray*}
  \hat{b}(s)&=&\mathcal{O}(\eps^{2\beta/3}),\\
  \hat{y}(s)&=&\mathcal{O}(\eps^{(2\beta+4)/3}),\\
  \hat{v}(s)&=&\mathcal{O}(\eps^{(2\beta+2)/3}),\\
  \hat{z}(s)&=&\mathcal{O}(\eps^{(2\beta+4)/3}),\\
  \hat{\lambda}_N(s)&=&\mathcal{O}(\eps^{(2\beta-2)/3}),
\end{eqnarray*}
when $s=\mathcal{O}(1)$. Recalling that $t=\eps^{2/3}s$, we obtain the 
theoretical predictions given in Table \ref{tab:exponents}.

\subsection{Interpretation for the dynamics}
\label{subsec:4.5}
To determine whether lift-off ($\lambda_N=0$) 
or impact-like behaviour ($\lambda_N$ large) occurs for the inner solution
we study
\begin{equation*}
  \lambda_N=\bar{\lambda}_N+\hat{\lambda}_N=\frac{\alpha_2^*}{\alpha_1^*-\alpha_3^*}-
  \frac{C_1}{\alpha_3^*}\kappa^{1-\beta}\eps^{(2\beta-2)/3}\frac{d\Theta}{d\tau}(\tau,\beta).
\end{equation*}
In $\hat{\lambda}_N$, we must consider that $\eps^{(2\beta-2)/3}$ is
very large or very small depending on whether $\beta<1$ or
$\beta>1$. If $\beta<1$, a sign change from positive to negative for
$C_1d\Theta/d\tau$ means lift-off, whereas a positive sign throughout
means impact. If $\beta>1$, the size of $d\Theta/d\tau$ must be
very large to have any effect.
From Theorems \ref{th:4} and \ref{th:5} in 
Appendix \ref{A:1},  we know that
\begin{equation}
  \frac{d\Theta}{d\tau}(\tau,\beta)\begin{cases}
  \text{is negative} &\text{ for $\tau$ large and negative,}\\
  \text{has sign of }\frac{\Gamma\left(\frac{1-\beta}{3}\right)}{\Gamma(-\beta)}&\text{ for $\tau=0$,}\\
  \text{is very large with sign of }\frac{1}{\Gamma(-\beta)}&\text{ for $\tau$ large and positive.}\\
  \end{cases}\label{eq:dTheta_sign}
\end{equation}
Here $Gamma$ represents the Gamma function, and we recall 
that $\Gamma(-\beta)$ is negative (respectively positive) whenever 
$\beta \in (2n-1,2n)$ for positive integers $n$
(respectively $\beta \in (2n-2,2n)$). Using this we can decide whether or not
impact of lift-off happens as $t$ passes through zero.

We are now in a position to piece together what happens to initial conditions
that in the outer scaling approach the $G$-spot. Let us treat two separate
cases.
\begin{description}
\item[Cases I and II, $0<\beta<1$.] In this case, the distinguished
  trajectory represents the strong stable manifold of the G-spot in
  the singular system.  Note that all trajectories of interest have
  $C_1<0$. From \eqref{eq:dTheta_sign} we can see that in the inner
  solution, $\hat{\lambda}_N>0$ for all three $\tau$
  regimes. Numerical computations of the derivative of $\Theta$
  supports that $\hat{\lambda}_N>0$ for all
  $\tau=\mathcal{O}(1)$. Further $\hat{\lambda}_N$ always dominates
  $\bar{\lambda}_N$ for small $\eps$.
Together, this is strong evidence that an impact must occur, although not quite a proof because it is possible that although $\Theta$ 
is very negative for large
$|\tau|$ and for $\tau=0$ it is conceivable that it might become 
positive for some intermediate $\tau$-value. We have found no numerical evidence
that such a possibility occurs. 

\item[Case III, $\beta >1$.] We now no longer need to limit ourselves
  to trajectories with $\hat{b}(\tau)<0$ for large negative $\tau$,
  and so we need to consider both possible signs of $C_1$ in
  \eqref{eq:bhat_inner}. Also $\bar{\lambda}_N$ dominates
  $\hat{\lambda}_N$ for small $\eps$ and $\tau=\mathcal{O}(1)$. Thus
  lift-off or impact is determined by the behaviour for $\tau$ large
  and positive in \eqref{eq:dTheta_sign}. We find that lift-off occurs
  when $C_1/\Gamma(-\beta)<0$, due to the very large growth rate of  
$\Theta$ (see Theorem \ref{th:5}). 
Note that $\tau_\text{lift-off}$ grows very
  slowly as $\eps\rightarrow0$, like $\log(\eps)^{2/3}$. 
For the other sign of $C_1\Gamma(-\beta)$, we are in the same situation as in 
above where there is strong evidence that an IWC occurs. 
The distinguished trajectory, obtainable
by setting $C_1=0$ is the dividing {\em canard} 
trajectory between lift-off and impact. Note
that $\Gamma(-\beta)$ changes sign whenever $\beta$ passes through
a positive integer. 
So there is a {\em side-switching} between which sign of perturbation
from the distinguished trajectory that leads to lift-off and which to
impact.
\end{description}

\section{Asymptotic analysis of a general system}
\label{sec:5}

Motivated by the previous example, we shall now consider
general systems of the form introduced in Sec.~\ref{sec:2}. 
We shall find that the hard part of the analysis is to determine the existence
and properties of a distinguished trajectory that is smooth in $t$ and
$\epsilon$. Once this is established, the behaviour of small deviations from
this trajectory will turn out to be precisely as in the motivating example. 

Again, we will assume positive slip and contact, that is
$\lambda_T=-\mu\lambda_N$ and $\eps^2\lambda_N=z-y(\xi)$. Further, we
will rename the functions $y(\xi)$, $v(\xi)$, $p(\xi)$, and $b(\xi)$
using the upper case variables $\yY(\xi)$, $\vV(\xi)$, $\pP(\xi)$, and
$\bB(\xi)$ instead.  
Then we consider $y$, $v$, $b$, and $p$ to be
additional scalar time-dependent variables, independent of $\xi$, that
extend the system, but satisfy the same differential equations as
$\yY(\xi)$, $\vV(\xi)$, $\pP(\xi)$, and $\bB(\xi)$ would.  To restore
the properties $p(t)=\pP(\xi(t))$ etc, we need only synchronise them
at one time point $t_0$. Thus our full system becomes
\begin{align}
  \dot{\xi} &= F(\xi) + G(\xi)\lambda_N, \label{eq:genSysXi}\\ 
  \dot{p}&=\alpha_1(\xi), \label{eq:genSysP}\\
  \dot{b}&=\alpha_2(\xi)-\alpha_3(\xi)\lambda_N,\label{eq:genSysB}\\
  \dot{y}&=v,\label{eq:genSysY}\\
  \dot{v}&=b+p\lambda_N, \label{eq:genSysV}\\
  \eps\dot{z}&=-z-\eps^2\lambda_N, \label{eq:genSysZ}
  \\ 
  \eps^2\lambda_N &= z-y, \label{eq:genSysLN}
\end{align}
with synchronisation conditions
\begin{equation}
y(t_0) = \yY(\xi(t_0)), \qquad v(t_0)=\vV(\xi(t_0)), \qquad p(t_0) = \pP(\xi(t_0)), \qquad b(t_0)=\bB(\xi(t_0)).\label{eq:genSysSynch}
\end{equation}

A G-spot $\xi^*$ is characterised by
\begin{equation}
  \yY(\xi^*)=0, \qquad \vV(\xi^*)=0, \qquad \pP(\xi^*)=0, \qquad
  \bB(\xi^*)=0,\label{eq:genSysGSpot}
\end{equation}
which are four (assumed to be independent) conditions on the
$m$-dimensional state $\xi$. To fix a particular G-spot, it is
convenient to introduce an $(m-4)$-dimensional additional system of equations
\begin{equation}
  \jJ(\xi^*)=0.\label{eq:genSysGSpotExtra}
\end{equation}
Local uniqueness of $\xi^*$ is guaranteed if we assume 
the non-degeneracy condition
that the $m$-dimensional Jacobian
\begin{equation}
[\pP_\xi(\xi^*),\bB_\xi(\xi^*),\yY_\xi(\xi^*),\vV_\xi(\xi^*),\jJ_\xi(\xi^*)]
\quad \mbox{is non-singular.} \label{eq:genSysGSpotNonSing}
\end{equation}

\subsection{An inner scaling}

We proceed very much as in in motivating example in Sec.~\ref{sec:4} by
adopting an inner time scale
$$
  dt=\delta^2 ds, \qquad \delta=\epsilon^{1/3}. 
$$
Note though that for the motivating example, the equation 
\eqref{eq:modelSysTEliminated} is linear in $b$, so it was not necessary to 
scale any of the dependent variables in the inner region. In general though
the system of equations \eqref{eq:genSysXi}-\eqref{eq:genSysLN} are nonlinear
in $\xi$. Therefore it is convenient to 
scale the dependent variables like
\begin{align*}
  \xi(t,\epsilon)&=\xi^* + \delta^2\tilde{\xi}(s,\delta)\\
  p(t,\epsilon)&=\delta^2\tilde{p}(s,\delta)\\
  b(t,\epsilon)&=\delta^2\tilde{b}(s,\delta)\\
  y(t,\epsilon)&=\delta^6\tilde{y}(s,\delta)\\
  v(t,\epsilon)&=\delta^4\tilde{v}(s,\delta)\\
  z(t,\epsilon)&=\delta^6\tilde{z}(s,\delta),
\end{align*}
where $\xi^*$ is the location of the $G$-spot and in what follows the
accent $\tilde{}$ 
will be exclusively used to represent these scaled inner variables. 
In the inner scale, the system becomes
\begin{align}
\tilde{\xi}^\prime &=F(\xi)+G(\xi)\lambda_N,\label{eq:specialX}\\
\tilde{p}^\prime &=\alpha_1(\xi),\label{eq:specialPP}\\
\tilde{b}^\prime&=\alpha_2(\xi)-\alpha_3(\xi)\lambda_N,\label{eq:specialB}\\
\tilde{y}^\prime&=\tilde{v},\label{eq:specialY}\\
\tilde{v}^\prime&=\tilde{b}+\tilde{p}\lambda_N,\label{eq:specialV}\\
\delta\tilde{z}^\prime&=-\tilde{z}-\lambda_N,\label{eq:specialZ}\\
\lambda_N&=\tilde{z}-\tilde{y},\label{eq:specialL}
\end{align}
with $^\prime=\frac{d}{ds}$. 

Let us couple the origin of the new time variable $s$ with the zero
value of $p$, by requiring that $p=0$ when $s=0$ for all $\delta$: 
\beq
\tilde{p}(0,\delta)=0.
\label{eq:pbc}
\eeq
Using the synchronisation condition \eqref{eq:genSysSynch} at $s=0$ we require for all $\delta$:
\begin{align}
  \pP\left(\xi^*+\delta^2\tilde{\xi}(0,\delta)\right)-\delta^2\tilde{p}(0,\delta)&=0 ,
\label{eq:specialInitialFirst}\\
  \bB\left(\xi^*+\delta^2\tilde{\xi}(0,\delta)\right)-\delta^2\tilde{b}(0,\delta)&=0, 
\label{eq:specialInitialSecond}\\
  \yY\left(\xi^*+\delta^2\tilde{\xi}(0,\delta)\right)-\delta^6\tilde{y}(0,\delta)&=0,  
\label{eq:specialInitialThird}\\
  \vV\left(\xi^*+\delta^2\tilde{\xi}(0,\delta)\right)-\delta^4\tilde{v}(0,\delta)&=0 . 
\label{eq:specialInitialFourth}
\end{align}
In addition, we will remove the $(m-4)$-dimensional freedom in the
location of the G-spot by imposing the additional boundary conditions
\eqref{eq:genSysGSpotExtra} on $\xi$ at $s=0$ for all $\delta$: 
\beq
\jJ(\xi^*+\delta^2\tilde{\xi}(0,\delta))=0,  
\label{eq:Jbc}
\eeq
where we still assume the non-degeneracy 
condition \eqref{eq:genSysGSpotNonSing}.

Note that the boundary conditions 
\eqref{eq:pbc}--\eqref{eq:Jbc} provide only
$m+1$ initial conditions to $m+5$ differential equations 
\eqref{eq:specialX}--\eqref{eq:specialZ}. Thus to find a
specific trajectory we have to specify four further conditions. 

\subsection{A distinguished trajectory}
\label{sec:distinguished}

The key step now is to establish that there is a unique distinguished
trajectory of the inner system of equations that 
plays the role of the explicit trivial solution of
the motivating example, that passes through the G-spot in the limit
$\eps \to 0$. This trajectory will have initial conditions that 
satisfy \eqref{eq:pbc}--\eqref{eq:Jbc}, which leaves four unspecified initial conditions.
Instead of four specific initial conditions, we will instead
require that $b(t,\eps)$ should be sufficiently smooth function that
it can be expressed as a regular power series in its
arguments up to arbitrary order. 

To motivate this requirement, 
we already know from the phase plane analysis of the singular 
system \eqref{eq:Gspot1}, \eqref{eq:Gspot2} in 
Fig.~\ref{fig:Gspot3} that when $\delta=0$ there are an open set
of initial conditions all of which pass through a particular
G-spot. Hence one of these initial conditions is essentially
required to fix a particular distinguished trajectory in the 
$(p,b)$-plane. Consider in particular case III, the
other two cases are somewhat more trivial. Looking at 
Fig.~\ref{fig:Gspot3} we note that all trajectories approach the $G$-spot
tangent to the weak stable eigenvector  
$\alpha_2 p = (\alpha_1-\alpha_3)b$. Then, according to recent results
in stable manifold theory (see \cite{Haller} and references therein),
of all the trajectories of the planar system 
\eqref{eq:Gspot1}-\eqref{eq:Gspot2},
there is a unique one whose graph $b(p)$ is smooth up to
order $C^{\beta+1}$ at $b=p=0$. 
Such a distinguished, maximally smooth trajectory is indicated by the 
dashed (red) line in Fig.~\ref{fig:Gspot3} in each of the three cases. 
The remaining three freedoms essentially arise by requiring that $y$, $z$, and $v$
are chosen so that there is additional smoothness in $t$ and $\eps$ so that
the trajectory in question does not blow up as $\eps \to 0$. 

We construct this maximally smooth trajectory as an asymptotic expansion. 
The procedure is a little involved, and 
makes use of special 
spaces of polynomials.  
Let $P_n$ be the polynomial space in $s$ spanned by 
$\{s^n,s^{n-3},s^{n-6},\ldots, s^q\}$, where $0\leq q= n$ $(\mbox{mod}\:  3)$. We shall also extend this definition by assuming that for $n<0$, $P_n$ consists
of the zero function. The result can be expressed as follows
\begin{theorem}
\label{th:1}
Let $\xi^*$ be a solution to equations
\eqref{eq:genSysGSpot}--\eqref{eq:genSysGSpotExtra} for which
the non-degeneracy condition \eqref{eq:genSysGSpotNonSing} holds.
Furthermore, let
$\alpha_1^*,\alpha_3^*<0$, evaluated at this $\xi^*$, be such that $\beta=\alpha_3^*/\alpha_1^*$ is not an 
integer. 
Then, 
there is a unique set of polynomial functions of $s$
$\xi_n(s) \in \Rset^m$ and $p_n(s),b_n(s),y_n(s),v_n(s),z_n(s),\lambda_n(s) 
\in \Rset$ for which
\begin{align*}
  \tilde{\bar{\xi}}(s,\delta)&=\left(\sum_{n=0}^{M-1}\xi_n(s)\delta^n\right)+\mathcal{O}(\delta^M),\\
  \tilde{\bar{p}}(s,\delta)&=\left(\sum_{n=0}^{M-1}p_n(s)\delta^n\right)+\mathcal{O}(\delta^M),\\
  \tilde{\bar{b}}(s,\delta)&=\left(\sum_{n=0}^{M-1}b_n(s)\delta^n\right)+\mathcal{O}(\delta^M),\\
  \tilde{\bar{y}}(s,\delta)&=\left(\sum_{n=0}^{M-1}y_n(s)\delta^n\right)+\mathcal{O}(\delta^M),\\
  \tilde{\bar{v}}(s,\delta)&=\left(\sum_{n=0}^{M-1}v_n(s)\delta^n\right)+\mathcal{O}(\delta^M),\\
  \tilde{\bar{z}}(s,\delta)&=\left(\sum_{n=0}^{M-1}z_n(s)\delta^n\right)+\mathcal{O}(\delta^M),\\
  \bar{\lambda}_N(s,\delta)&=\left(\sum_{n=0}^{M-1}\lambda_n(s)\delta^n\right)+\mathcal{O}(\delta^M),
\end{align*}
satisfy 
equations (\ref{eq:specialX}--\ref{eq:specialL}) 
up to order $\mathcal{O}(\delta^M)$ and equations
(\ref{eq:pbc}--\ref{eq:Jbc}) up to order $\mathcal{O}(\delta^{M+2})$.  More specifically, 
these functions belong to the spaces
$$
p_{2\nu}(s),b_{2\nu}(s),\xi_{2\nu}(s)\in P_{\nu+1}, \quad 
y_{2\nu}(s),z_{2\nu}(s)\in P_\nu, \quad 
v_{2\nu}(s)\in P_{\nu-1}, \quad
\lambda_{2\nu}(s)\in P_\nu
$$
for even powers $n=2\nu$ of $\delta$, and
$$
p_{2\nu+1}(s),b_{2\nu+1}(s),\xi_{2\nu+1}(s)\in P_{\nu-3}, \quad 
y_{2\nu+1}(s),z_{2\nu+1}(s)\in P_{\nu-1}, \quad 
v_{2\nu+1}(s)\in P_{\nu-2}, \quad
\lambda_{2\nu+1}(s)\in P_{\nu-4}
$$
for odd powers $n=2\nu+1$ of $\delta$.

\end{theorem}

In what follows,  
for functions of $\xi$, like $F$, $\alpha_2$, or $\vV$, it is 
useful to introduce a notation for the coefficients in a $\delta$ expansion:
\begin{align}
f(\bar{\xi})=\sum_{k=0}^nf_k(s)\delta^k+\mathcal{O}(\delta^{n+1}).
\label{eq:f_kdef}
\end{align}
We also define $f_k(s)=0$ for all $k<0$.
Note that the use of the index 
$k$ in $f_k(s)$ is equivalent to that used for the {\em scaled} variables
like $\tilde{\bar{\xi}}$, $\tilde{\bar{p}}$, or $\tilde{\bar{v}}$. 
If it were to be
applied to the unscaled variables, the index would be different.
For example $v_k(s)$ is the
coefficient of $\delta^k$ in an expansion of $\tilde{v}$, but the
coefficient of $\delta^{k+4}$ in an expansion of $v$ itself, whereas $f_k$ is
always the coefficient of $\delta^k$ for a function $f(\bar{\xi})$. 

We begin by stating a useful result:

\begin{lemma}\label{thm:expansion}
  Assume $f(\xi)$ is a $C^n$ function. Then $f_n(s)$ only depends on
  $\xi_k(s)$ for $0\leq k\leq n-2$, and $\xi_{n-2}$ enters linearly
  with coefficient $f_\xi(\xi^*)$.
  
  Assume further that $\xi_{2k}(s)\in P_{k+1}$, $\xi_{2k+1}(s)\in
  P_{k-3}$. Then if $n=2\nu$ then $f_{n}(s)\in P_{\nu}$, and if $n=2\nu+1$
  then  $f_{n}(s)\in P_{\nu-4}$
\end{lemma}

\begin{proof}
  The first part is immediate through Taylor expansion in $\delta$.

For the second part, note to begin with that the product of a polynomial in
  $P_k$ and one in $P_l$ is in $P_{k+l}$. Furthermore, 
note that $f_n(s)$ is a sum of products, each product being a
product of a constant and some $\xi_{k_i}(s)$, where
$n=\sum_i(k_i+2)$. We consider the cases $n$ even and odd separately. 

First consider the case of even $n$, specifically $n=2\nu$. If all
$k_i$ are even so $k_i=2\kappa_i$, then the term is a product of a
constant and some $\xi_{2\kappa_i}(s)$ each of which is in
$P_{\kappa_i+1}$ and thus the term is in $P_{\sum_i(\kappa_i+1)}$. But
$2\nu=\sum_i(2\kappa_i+2)=2\sum_i(\kappa_i+1)$ so the term is in
$P_\nu$. If instead two of the $k_i$, say $k_1=2\kappa_1+1$ and
$k_2=2\kappa_2+1$, are odd and the rest even: $k_i=2\kappa_i$ for
$i>2$, then the term is in $P_{\kappa_1-3+\kappa_2-3+\sum_{i>2}(\kappa_i+1)}$, but
$2\nu=2\kappa_1+3+2\kappa_2+3+\sum_{i>2}(2\kappa_i+2)=2(\kappa_1-3+\kappa_2-3+\sum_{i>2}(\kappa_i+1))+18$ so the term is in $P_{\nu-9}$ which 
is included in $P_{\nu}$. In the
same way, each time there are two new odd $k_i$, the resulting term
order is lowered by 9.

Second, consider the case of odd $n$, specifically $n=2\nu+1$. If
there is only one odd $k_i$, say $k_1=2\kappa_1+1$ and the rest even
$k_i=2\kappa_i$ for $i>1$, then the term is in
$P_{\kappa_1-3+\sum_{i>1}(\kappa_i+1)}$, but
$2\nu+1=2\kappa_1+3+\sum_{i>1}(2\kappa_i+2)=
2(\kappa_1-3+\sum_{i>1}(\kappa_i+1))+9$ so the term is in
$P_{\nu-4}$. Again, each time two more $k_i$ are odd, the term order
is lowered by 9, and is included in $P_{\nu-4}$.

\end{proof}

\begin{proof}[Proof of Theorem \ref{th:1}]
To establish the expansion, we set up an iteration scheme to compute
the solution at order $n$ in terms of the solutions at orders less
than $n$. The iteration scheme works as follows. Let $n\geq 0$. If $n>0$ then suppose that solutions for 
$p_k$, $b_k$, $y_k$, $v_k$, $z_k$, $\lambda_k$ and $\xi_k$ have been computed for all 
$0\leq k\leq n-1$ and they belong to the appropriate polynomial spaces as specified by the theorem. 
We then find a solution at $\mathcal{O}(\delta^n)$ through the following steps.
\begin{enumerate}
\item Consider the order $\delta^n$ term of both the differential
  equation \eqref{eq:specialPP} and
  the initial condition 
  \eqref{eq:pbc}.
  This gives
  \begin{equation}
    p_n^\prime={\alpha_1}_n
    \label{eq:p_nprime}
  \end{equation}
  and $p_n(0)=0$, where
owing to Lemma \ref{thm:expansion}, 
the right-hand side is a known polynomial of $s$ in $P_{\nu}$ if $n=2\nu$
or $P_{\nu-4}$ if $n=2\nu+1$. Integrating \eqref{eq:p_nprime} yields
a unique $p_{2\nu} \in P_{\nu+1}$ or $p_{2\nu+1}\in P_{\nu-3}$. For example, we get $p_0=\alpha_1^*s$ for all systems.

\item Consider the term of order $\delta^n$ in \eqref{eq:specialB},
  which can be written
  \begin{equation}
    b_n^\prime+\alpha_3^*\lambda_n
={\alpha_2}_n-\sum_{k=0}^{n-1}{\alpha_3}_{n-k}\lambda_k  
:= =r_{b,n}(s),
\label{eq:specialBn}
  \end{equation}
where the right-hand side $r_{b,n}(s)$ is a known function.
Using Lemma \ref{thm:expansion} and the known polynomial form of
$\lambda_k$, we find $r_{b,2\nu}\in P_\nu$ and  $r_{b,2\nu+1}\in P_{\nu-4}$.
  
Similarly we can write the order $\delta^n$ term of \eqref{eq:specialY} as
  \begin{equation}
    y_n^\prime-v_n=0.\label{eq:specialYn}
  \end{equation}
The order $\delta^n$ term of \eqref{eq:specialV} can be written
  \begin{equation}
    v_n^\prime-b_n-\alpha_1^*s\lambda_n
=\sum_{k=0}^{n-1}p_{n-k}\lambda_k  : =r_{v,n}(s). \label{eq:specialVn}
  \end{equation}
Here we have used $p_0=\alpha_1^*s$ from the very first step, and
note that we need $p_n$ from step 1.
The known right hand side now found to be $r_{v,2\nu}(s)\in P_{\nu+1}$
or $r_{v,2\nu+1}(s)\in P_{\nu-3}$.

The order $\delta^n$ term of \eqref{eq:specialZ} and \eqref{eq:specialL}
can be written
   \begin{equation}
    -z_n-\lambda_n=z_{n-1}^\prime\label{eq:specialZn}
   \end{equation}
and
   \begin{equation}
    \lambda_n-z_n+y_n=0\label{eq:specialLn}.
   \end{equation}
Note that \eqref{eq:specialZn} remains true if $n=0$ since we 
have defined $z_{-1}(s)=0$. 
   
So far we have obtained a system of four coupled ODEs 
(\ref{eq:specialBn}--\ref{eq:specialZn}) and one algebraic equation
\eqref{eq:specialLn} for the unknowns 
$b_n$, $\lambda_n$, $y_n$, $z_n$ and $v_n$. Next, we eliminate four of these variables one by one. First, differentiate \eqref{eq:specialYn}, insert the result
into \eqref{eq:specialVn}, multiply the resulting equation by
$\alpha_3^*$ and finally eliminate $\lambda_n$ using
\eqref{eq:specialBn} to get
   \begin{equation}
     \alpha_3^*y_n^{\prime\prime}-\alpha_3^*b_n+\alpha_1^*s
     b_n^\prime=\alpha_1^*sr_{b,n}+\alpha_3^*r_{v,n}.\label{eq:specialYnbis}
   \end{equation}

Next, we can eliminate $z_n$ and $\lambda_n$ from
   \eqref{eq:specialBn}, \eqref{eq:specialZn}, and
   \eqref{eq:specialLn} to get
  \begin{equation*}
    \alpha_3^*y_n-2 b_n^\prime=-\alpha_3^*z_{n-1}^\prime-2r_{b,n}.
  \end{equation*}
Differentiating twice and using the result to eliminate
$y_n^{\prime\prime}$ from \eqref{eq:specialYnbis} finally gives us
    \begin{equation}
    2b_n^{\prime\prime\prime}+\alpha_1^*sb_n^\prime-\alpha_3^*b_n
=2r_{b,n}^{\prime\prime}+\alpha_3^*z_{n-1}^{\prime\prime\prime}+\alpha_1^*sr_{b,n}+\alpha_3^*r_{v,n}:=r_n(s).\label{eq:third-order_b}
  \end{equation}
The polynomial order for the known right-hand side can now be found
to be $r_{2\nu}\in P_{\nu+1}$ or $r_{2\nu+1}\in P_{\nu-3}$.

Now, note that \eqref{eq:third-order_b} is a linear inhomogeneous equation.
The solution is in general composed of a complementary function 
plus a particular solution. But we know by Theorem \ref{th:3} (see 
Appendix \ref{A:1}) 
that if $\beta$ is not an integer, the complementary function 
is a linear combination of generalised hypergeometric functions 
in the rescaled variables $s$, $\delta$, which does not
satisfy the required smoothness assumptions. Therefore, we must take the
particular solution only.

Substituting a monomial $s^k$ for $b_n$ into the left-hand side of
\eqref{eq:third-order_b} gives
\begin{equation*}
  (k\alpha_1^*-\alpha_3^*)s^k+k(k-1)(k-2)s^{k-3}
\end{equation*}
Since we have assumed $\beta=\alpha_3^*/\alpha_1^*$ is not an integer, the
coefficient of $s^k$ is non-zero. This means we can make an ansatz
$b_{2\nu}\in P_{\nu+1}$ or $b_{2\nu+1}\in P_{\nu-3}$ and find its coefficients one by one starting with
the highest order. 

Thus there is a unique particular integral solution with 
$b_{2\nu}\in P_{\nu+1}$ or $b_{2\nu+1}\in P_{\nu-3}$.

\item Having found $b_n$, we can recover $y_n$, $v_n$, $z_n$ and
  $\lambda_n$ from
  \begin{align*}
    \alpha_3^*y_n&=2(b_n^\prime-r_{b,n})-\alpha_3^*z_{n-1}^\prime,\\
    \alpha_3^*v_n&=2(b_n^{\prime\prime}-r_{b,n}^\prime)-\alpha_3^*z_{n-1}^{\prime\prime},\\
    \alpha_3^*\lambda_n&=r_{b,n}-b_n^\prime,\\
    \alpha_3^*z_n&=b_n^\prime-r_{b,n}-\alpha_3^*z_{n-1}^\prime.
  \end{align*}
By studying the right-hand sides for even and odd $n$, we can verify that
$y_n$, $v_n$ and $z_n$ are in the correct polynomial spaces.

\item Finally, consider the order $\delta^n$ term in the
  differential equation \eqref{eq:specialX}, which gives
  \begin{equation}
    \xi_n^\prime=r_{\xi,n}(s)=F_n+\sum_{k=0}^{n}G_{n-k}\lambda_n,
    \label{eq:xinprime}
  \end{equation}
where $r_{\xi,2\nu}\in
  P_{\nu}$ or  $r_{\xi,2\nu+1}\in P_{\nu-4}$. 
Note we need $\lambda_n$ from step 3 here. 
We obtain an explicit expression for $\xi_n$ by 
integrating both sides of \eqref{eq:xinprime}. The integration constants 
will be eliminated with the help of the order $\delta^{n+2}$
terms in the $m$-dimensional
initial conditions \eqref{eq:specialInitialFirst}--\eqref{eq:Jbc}:
  \begin{equation}
    [\pP_{n+2}(0)-p_n(0),\bB_{n+2}(0)-b_n(0),\yY_{n+2}(0)-y_{n-4}(0),\vV_{n+2}(0)-v_{n-2}(0),\jJ_{n+2}(0)]=[0,0,0,0,0]
    \label{eq_initialcond_xi}
  \end{equation}
According to Lemma \ref{thm:expansion}, $\pP_{n+2}(0)$, $\bB_{n+2}(0)$, $\yY_{n+2}(0)$ and $\vV_{n+2}(0)$ depend only on $\xi_k(0)$ for $0\leq k\leq n$, furthermore \eqref{eq_initialcond_xi} can be rearranged to read
\begin{equation}
    [\pP_\xi(\xi^*),\bB_\xi(\xi^*),\yY_\xi(\xi^*),\vV_\xi(\xi^*),\jJ_\xi(\xi^*)]\xi_n(0)=r_{\xi_0,n},
\label{eq:xi_iteraction}
  \end{equation}  
   where the left-hand side is a linear in
$\xi_n(0)$ (see Lemma \ref{thm:expansion}). 
The right-hand side $r_{\xi_0,n}$ is then an $m$-vector, each component of
which contains a a sum of two types of terms: (i) constants times the products of 
lower-order terms $\xi_k(0)$ ($k<n$) and; (ii) terms that involve 
$p_{n}(0)$, $b_{n}(0)$, $y_{n-4}(0)$, $v_{n-2}(0)$. 
  
If we treat $s$ as a free variable in the terms of type (i) 
(instead of having $s=0$), then each of them belongs to the polynomial class 
$P_{\nu+1}$ if $n=2\nu$ or $P_\nu-3$ if $n=2\nu+1$. 
This result can be proven in the same way as the second 
statement of Lemma \ref{thm:expansion}, which relies 
on the known polynomial class of $\xi_k$ for $k<n$. It follows that the 
polynomials (i) do not include zeroth-order terms and thus their values 
for $s=0$ are 0, unless $\nu\pmod{3}=2$ 
and $n=2\nu$ or $\nu\pmod{3}=0$ and $n=2\nu+1$.

The functions $p_n$, $b_n$, $y_{n-4}$ and $v_{n-2}$ appearing in terms
of type (ii) also belong to special polynomial classes as specified by
the statement of the theorem, and as verified in previous steps of the
iteration scheme.  It follows that the constant terms of these
polynomials must vanish, and thus their values for $s=0$ are 0 for the
exact same values of $n$ where the terms of type (i) also vanish.
  
Hence, we have found that $r_{\xi_0,2\nu}$ are all
zero unless  $\nu\pmod{3}=2$ and $r_{\xi_0,2\nu+1}$ are all
zero unless  $\nu\pmod{3}=0$. At the same time, the system matrix on
the left-hand side of \eqref{eq:xi_iteraction} is non-singular 
by the assumption of the theorem. This implies
$\xi_{2\nu}(0)$ or $\xi_{2\nu+1}(0)$ is well defined, and is zero
unless 
$\nu\pmod{3}=2$ or $\nu\pmod{3}=0$, respectively. 
Hence we have found the auxiliary conditions for \eqref{eq:xinprime}
and we can conclude that the integration of \eqref{eq:xinprime} yields
a unique $\xi_{2\nu}\in P_{\nu+1}$ or $\xi_{2\nu+1}\in P_{\nu-3}$. It
is worth noting that for all values of $n$ for which the polynomial
class $P_{\nu+1}$ (even $n$) or $P_{\nu-3}$ (odd $n$) does not include
constant functions, the previously described procedure obtains the
initial condition $\xi_{n}(0)=0$, thus eliminating the integration constant.
\end{enumerate}
\end{proof}

\begin{corollary}
The polynomial classes established by 
Theorem \ref{th:1} imply that each of the unscaled variables 
$\xi$, $p$, $b$, $y$ and $v$ truncated to any finite
order in $\delta$ can be written as a polynomial in $t,\epsilon$.
Hence we can express the distinguished smooth trajectory as a regular 
asymptotic expansion in $\eps$. 
\end{corollary}

\begin{proof}
We just demonstrate that the statement is true for $p$. 
The construction for the other variables 
is similar. 
Note  that 
the formula for 
$\tilde{\bar{p}}(s,\delta)$ in the theorem
consists of a sum of terms like
$$
 \underbrace{p_{2\nu}(s)}_{\in P_{\nu+1}}
 \delta^{2\nu} \;  \text{ and } \;
 \underbrace{p_{2\nu+1}(s)}_{\in P_{\nu-3}}
 \delta^{2\nu+1} 
$$
from which the rescaled version of this variable is a sum of terms like
$$
 \underbrace{p_{2\nu}(s)}_{\in P_{\nu+1}}
 \delta^{2\nu+2} \;  \text{ and } \;
 \underbrace{p_{2\nu+1}(s)}_{\in P_{\nu-3}}
 \delta^{2\nu+3} 
$$
or equivalently
$$
\sum_{\rho=\nu+1,\nu-2,\nu-5,...}
 Ks^\rho\delta^{2\nu+2}
 \;\text{ and }\;
\sum_{\rho=\nu-3,\nu-6,\nu-9,...}
   K s^\rho
 \delta^{2\nu+3}
$$
where $K$ represents any unspecified constant.
Replacing $s$ by $t$ and $\delta$ by $\eps$, these two terms become
\begin{align*}
\sum_{\rho=\nu+1,\nu-2,...}  K
t^\rho
 \delta^{2\nu-2\rho+2} 
 =
 \sum_{0\leq\sigma\leq(\nu+1)/3}
 K \delta^{6\sigma}t^{\nu+1-3\sigma}
 =
 \sum_{0\leq\sigma\leq(\nu+1)/3}
 K \eps^{2\sigma}t^{\nu+1-3\sigma}
\end{align*}
and
\begin{align*}
\sum_{\rho=\nu-3,\nu-6,...}
   K t^\rho
 \delta^{2\nu-2\rho+3}
 =
 \sum_{0\leq\sigma\leq(\nu-3)/3}
   K \delta^{9+6\sigma}t^{\nu-3-3\sigma}
   =
 \sum_{0\leq\sigma\leq(\nu-3)/3}
   K \eps^{3+2\sigma}t^{\nu-3-3\sigma}  
\end{align*}
respectively, which are regular polynomials in $\eps,t$.
\end{proof}

\paragraph*{Example.} For the extended example system 
with $\alpha_1^*=\alpha_2^*=-1$ and
$\alpha_3^*=-3/2$, we find
\begin{align*}
\tilde{\bar{p}}(s,\delta) = & 
-s+{s}^{2}\chi\,{\delta}^{2}+2\,{s}^{3}{\chi}^{2}{\delta}^{4}+
 \left( 3\,{s}^{4}{\chi}^{3}-96\,s{\chi}^{3} \right) {\delta}^{6}
+\mathcal{O}(\delta^{8}) 
\\
\tilde{\bar{b}}(s,\delta)= & 
2\,s+6\,\chi\,{s}^{2}{\delta}^{2}+ \left( 12\,{\chi}^{2}{s}^{3}-96\,{
\chi}^{2} \right) {\delta}^{4}+ \left( {\frac {138\,{\chi}^{3}{s}^{4}
}{5}}-{\frac {10368\,s{\chi}^{3}}{5}} \right) {\delta}^{6}
+\mathcal{O} \left( {\delta}^{8} \right) 
\\
\tilde{\bar{y}}(s,\delta)= &
-4-16\,\chi\,s{\delta}^{2}+8\,\chi\,{\delta}^{3}-48\,{\chi}^{2}{s}^{2
}{\delta}^{4}+48\,s{\chi}^{2}{\delta}^{5}+ \left( -{\frac {736\,{\chi}
^{3}{s}^{3}}{5}}+{\frac {13824\,{\chi}^{3}}{5}}-48\,{\chi}^{2}
 \right) {\delta}^{6}
+\mathcal{O}\left( {\delta}^{7} \right)\\
\tilde{\bar{v}}(s,\delta)= & 
-16\,\chi\,{\delta}^{2}-96\,s{\chi}^{2}{\delta}^{4}+48\,{\chi}^{2}{
\delta}^{5}-{\frac {2208\,{s}^{2}{\chi}^{3}}{5}}{\delta}^{6}
+\mathcal{O}\left( {\delta}^{7} \right)
\\
\tilde{\bar{z}}(s,\delta)= &
-2-8\,s\chi\,{\delta}^{2}+8\,\chi\,{\delta}^{3}-24\,{s}^{2}{\chi}^{2}
{\delta}^{4}+48\,s{\chi}^{2}{\delta}^{5}+ \left( -{\frac {368\,{s}^{3}
{\chi}^{3}}{5}}+{\frac {6912\,{\chi}^{3}}{5}}-48\,{\chi}^{2} \right) {
\delta}^{6}
+\mathcal{O}\left( {\delta}^{7} \right)
\end{align*}
\begin{figure}
\begin{center}
\includegraphics[width=0.5\textwidth]{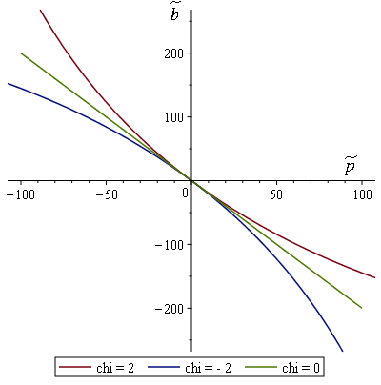}
\end{center}
\caption{Comparison between expansions for the distinguished trajectories
for $\eps=10^{-5}$.}
\label{fig:PainlevePeter}
\end{figure} 
Note that by construction, setting $\chi=0$ reconstructs the trivial
solution \eqref{eq:modelSysDistP}--\eqref{eq:modelSysDistV}. Figure \ref{fig:PainlevePeter} compares the solutions in the $(b,p)$-plane for different values of
$\chi$. 

To demonstrate that each of the above expressions implies that the
corresponding unscaled variable is a polynomial in $t$
and $\eps$, consider for example the expansion for $\bar{z}=\delta^6\tilde{\bar{z}}$
in the case $\chi=1$
under the substitution $s = \eps^{-2/3} t$ and $\delta=\eps^{1/3}$.
We have
\begin{align*}
\bar{z}(t,\eps)=\delta^6\tilde{\bar{z}}(s,\delta)= &
\delta^6\left(-2-8s{\delta}^{2}+8 {\delta}^{3}-24\,{s}^{2}
{\delta}^{4}+48s{\delta}^{5}+ 
\left( -\frac {368\,{s}^{3}}{5}+\frac{6672}{5} 
\right) {\delta}^{6}
+\mathcal{O}\left( {\delta}^{7} \right)\right)\\
= & \eps^2\left(-2 - 8t + 8 \eps - 24t^2 + 48 t \eps + 
\left( -\frac {368\,{t}^{3}}{5 \eps^2}+\frac{6672}{5} 
\right) \epsilon^2  +\mathcal{O}\left( {\delta}^{7} \right)\right)\\
= & \eps^2\left(-2 - 8 t - 24t^2 - \frac{368}{5} t^3 + (8 + 48 t)\eps 
+ \frac{6672}{5} \eps^2 + \mathcal{O}( t^4,\eps t^2, \eps^2 t, \eps^3)\right) 
\end{align*} 

\bigskip

Note that the distinguished trajectory exists for both $t<0$ and $t>0$ and
so can correspond to a canard solution that passes between 
the critical (slow) manifolds for $p<0$ and $p>0$. This solution can they play
the role of the separatrix in the inner system 
that separates trajectories that lift off
from those that take an IWC.  To see whether this is the case,
we have to consider other trajectories that are 
in the critical manifold for $p<0$. In order to do this we need to look
at the outer scale and consider the asymptotic behaviour as $p\to 0$ of
solutions in the slow manifold.   

\subsection{Fast-slow analysis of the outer system}

Consider the general system
\eqref{eq:genSysXi}--\eqref{eq:genSysLN}. 
Letting $\eps^2y_o=y$, $\eps^2z_o=z$, $\eps v_o=v$ gives
\begin{align*}
  \dot{\xi}&=F(\xi)+G(\xi)\lambda_N,\\
  \dot{p}&=\alpha_1(\xi),\\
  \dot{b}&=\alpha_2(\xi)-\alpha_3(\xi)\lambda_N,\\
  \eps\dot{y}_o&=v_o,\\
  \eps\dot{v}_o&=b+p\lambda_N,\\
  \eps\dot{z}_o&=-z_o-\lambda_N,\\
  \lambda_N &= z_o-y_o.
\end{align*}
Note that this is a fast-slow system.
 
\paragraph{The fast system} 
is obtained by letting $\dot{\xi}=\dot{p}=\dot{b}=0$,
in which case $\xi$, $p$ and $b$ are constant and we are left with a linear
system for the remaining three variables 
$$
\eps (\dot{y}_o,\dot{v}_o,\dot{z}_o)^T  
= M (y_o,v_o,z_o)^T,   
$$
where
$$
\mathbf{M}=
\left[\begin{array}{ccc}
0&1&0\\
-p&0&p\\
1&0&-2
\end{array}\right]
$$ 
The characteristic polynomial of $M$ is
$$
\lambda^3+2\lambda^2+p\lambda+p=0,
$$ 
which is the same as that
of the fast outer system in the motivating example, with the same
conclusions regarding stability.  Specifically, for $p>0$ trajectories
are attracted to a codimension three manifold, representing the slow
dynamics, whereas for $p<0$ the slow dynamics is normally hyperbolic
with a two-dimensional stable manifold and one-dimensional unstable
manifold.

\paragraph{The slow dynamics} for $\eps=0$ occur on the slow manifold 
\begin{align*}
  y_o&=2b/p,\\
  v_o&=0,\\
  z_o&=b/p,
\end{align*}
whose dynamics are given by the slow subsystem
\begin{align*}
  \dot{\xi}&=F(\xi)-G(\xi)b/p,\\
  \dot{p}&=\alpha_1(\xi),\\
  \dot{b}&=\alpha_2(\xi)+\alpha_3(\xi)b/p.
\end{align*}

Now, according to Fenichel theory (see \cite{Kuhn}), for all $p$ bounded
away from zero (where the slow manifold is normally-hyperbolic) then there
exist a {\em critical manifolds} which are $O(\eps)$ close to the slow
manifold for $p>0$ and $p<0$, are smooth and inherent the stability properties
of the slow manifold in each case. 

In order to understand the limit as $p\to 0$ of
the dynamics in the slow subsystem, it is useful to rewrite it in the form 
\begin{align*}
  \dot{\xi}-\left(F(\xi)+G(\xi)\lambda_N\right)&=0,\\
  \dot{p}-\alpha_1(\xi)&=0,\\
  \dot{b}-\left(\alpha_2(\xi)-\alpha_3(\xi)\lambda_N\right)&=0,\\
  -\left(b+p\lambda_N\right)&=0.
\end{align*}
We also write slow variables as deviations from the distinguished
trajectory 
\begin{align}
  \xi&=\bar{\xi}(t,0)+\hat{\xi}(t),\\
  p&=\bar{p}(t,0)+\hat{p}(t),\\
  b&=\bar{b}(t,0)+\hat{b}(t),\\
  \lambda_N&=\bar{\lambda}_N(t,0)+\hat{\lambda}_N(t).
\end{align}
Motivatived by the example system in \ref{sec:4}, we seek 
an ansatz of the form
\begin{align}
  \hat{\xi}(t)&=\xi_0(-t)^r+o((-t)^{r}),\\
  \hat{p}(t)&=p_0(-t)^{r+1}+o((-t)^{r+1}),\\
  \hat{b}(t)&=b_0(-t)^r+o((-t)^{r})\\
  \hat{\lambda}_N(t)&=\lambda_0(-t)^{r-1}+o((-t)^{r-1}),
\end{align}
for an unknown exponent $r>0$, and using $\bar{p}(t,0)=\alpha_1^*t+\cdots$,
we find to leading order that 
\begin{align}
  \left[-r\xi_0-G^*\lambda_0\right](-t)^{r-1}&=o((-t)^{r-1})\\
  \left[-(r+1)p_0-{\alpha_1}_\xi^*\xi_0\right](-t)^{r}&=o((-t)^{r})\\
  \left[-rb_0+\alpha_3^*\lambda_0\right](-t)^{r-1}&=o((-t)^{r-1})\\
  \left[-b_0+\alpha_1^*\lambda_0\right](-t)^{r}&=o((-t)^{r}).
\end{align}
From the last two equations, a solutions with $b_0\neq 0$ requires
$r=\alpha_3^*/\alpha_1^*=\beta$. Then we find
\begin{align}
  \xi_0&=-\frac{G^*}{\alpha_3^*}b_0,\\
  p_0&=0,\\
  \lambda_0&=\frac{1}{\alpha_1^*}b_0.
\end{align}
Thus
\begin{equation}
  b=\bar{b}(t,0)+b_0(-t)^{\beta}+o((-t)^{\beta}).
\label{eq:outerb_gen}
\end{equation}

\subsection{Matching the inner and outer solutions}

The inner system is given by
\eqref{eq:specialX}--\eqref{eq:specialL}. There is a fast timescale
$\delta$ (in $s$ time units). The fast dynamics is one-dimensional and $\tilde{z}$ 
evolves quickly to the slow manifold is $\tilde{z}=\tilde{y}/2$.
The slow system becomes
\begin{align*}
  \tilde{\xi}^\prime&=\tilde{F}(\xi)+\tilde{G}(\xi)\lambda_N\\
  \tilde{p}^\prime&=\alpha_1(\xi)\\
  \tilde{b}^\prime&=\alpha_2(\xi)-\alpha_3(\xi)\lambda_N\\
  \tilde{y}^\prime&=\tilde{v}\\
  \tilde{v}^\prime&=\tilde{b}+\tilde{p}\lambda_N,
\end{align*}
with $\lambda_N=-\tilde{y}/2$.

Motivated by the preliminary simulations of Sec.~\ref{sec:aymptoticsim}, 
we will look  for solutions (in the form of hatted variables) 
that are scaled deviations from the distinguished trajectory 
of the form
\begin{align}
  \tilde{p} & =\tilde{\bar{p}}(s,\epsilon^{1/3}) \nonumber\\
  \tilde{b}&=\tilde{\bar{b}}(s,\epsilon^{1/3})+\epsilon^{2(\beta-1)/3}\tilde{\hat{b}}(s) \label{eq:bpert} \\
  \tilde{y}&=\tilde{\bar{y}}(s,\epsilon^{1/3})+\epsilon^{(2(\beta-1)/3}\hat{y}(s)
\nonumber \\
  \tilde{v}&=\epsilon^{4/3}\tilde{\bar{v}}(s,\epsilon^{1/3})+\epsilon^{(2(\beta-1)/3}\hat{v}(s).
\nonumber 
\end{align}

Then, in the limit $\epsilon\to 0$, to leading order in $\xi$, using
the fact that $\tilde{\bar{p}} = \alpha_1^*s+ \ldots$, 
we get 
\begin{align*}
  \hat{b}^\prime&=\alpha_3^*\hat{y}/2,  \\
  \hat{y}^\prime&=\hat{v},\\
  \hat{v}^\prime&=\hat{b}-\alpha_1^*s\hat{y}/2.
\end{align*}
Elimination of $\hat{y}$ and $\hat{v}$ gives
\begin{equation*}
  \hat{b}^{\prime \prime \prime}+\frac{\alpha_1^*}{2}s \hat{b}^\prime-\frac{\alpha_3^*}{2}\hat{b}=0.
\end{equation*}
Rescaling time to $\tau= \kappa s$ with   $\kappa=(-\alpha_1^*/2)^{1/3}$, 
we get precisely the same equation \eq{eq:hypergeomeqn} that we obtained
for perturbations to the distinguished trajectory for the 
as we obtained for the example system in Sec.~\ref{sec:4}
whose asymptotics are summarised in Appendix \ref{A:1}.
The rest of the analysis of the dynamics of this equation follows
exactly as in Sec.~\ref{subsec:4.4}.
In particular, matching with the outer equation 
\eqref{eq:outerb_gen} shows that
\begin{equation*}
  \hat{b}(s)=b_0\kappa^{-\beta}\eps^{2\beta/3}\Theta(\kappa s,\beta), 
\end{equation*}
where the initial constant $b_0$ determines the sign of the perturbation from
the distinguished trajectory.   
Thus, applying the results from the Appendix on the 
asymptotics of hypergeometric functions, we get the same 
conditions \eqref{eq:dTheta_sign} that determine whether lift-off or
IWC occur.

Moreover, the implications for the dynamics are precisely as discussed in
Sec.~\ref{subsec:4.5}. 

\section{Application to a frictional impact oscillator}
\label{sec:6}

We now apply the previously developed theory to a frictional impact
oscillator proposed by \cite{leine2008}, see also
Fig. \ref{fig:examples}(b). Our goal here is to verify that the
approximate solutions produced by the expansion scheme of
Sec. \ref{sec:5} match the results of brute-force numerical
simulation.

\subsection{The system}

The frictional impact oscillator consists of two point masses, two springs and two dampers. The mass $m_1$ is in unilateral contact with a moving belt with friction coefficient $\mu$. The system has two mechanical degrees of freedom and thus we use the generalized coordinates
\begin{equation*}
  q=\begin{pmatrix}
  \phi \\ \psi
  \end{pmatrix}
\end{equation*}
As \cite{leine2008} shows, its motion is governed by the equation
$$
M(q)\ddot q=
f(q,\dot{q})+
Q_N(q)\lambda_N+Q_T\lambda_T
$$
with
\begin{eqnarray*}
  M(q)&=&\begin{pmatrix}
  m_1l^2 & m_1l\sin(\phi) \\ m_1l\sin(\phi) & m_1+m_2
  \end{pmatrix}, \\
  f(q,\dot{q})&=&\begin{pmatrix}
  -k_\phi(\phi-\phi_0)-c_\phi\dot{\phi}-m_1gl\sin(\phi) \\
  -k_\psi\psi-c_\psi\dot{\psi}-(m_1+m_2)g-m_1l\cos(\phi)\dot{\phi}^2
  \end{pmatrix}\\
  Q_T(q)&=&(\partial x/\partial q)^T=\begin{pmatrix}
  l\cos(\phi) \\ 0
  \end{pmatrix}, \\
  Q_N(q)&=&(\partial y/\partial q)^T=\begin{pmatrix}
  l\sin(\phi) \\ 1
  \end{pmatrix} \\
  \end{eqnarray*}
The horizontal and vertical position functions of the contact point are
\begin{equation*}
  x(q)=l\sin(\phi),\quad y(q)=\psi+l(1-\cos(\phi)).
\end{equation*}
Assuming positive slip, (i.e. $\lambda_T=-\mu\lambda_N$), these equations can be written in the form of \eqref{eq:genx} with  
\begin{align*}
\xi=\begin{pmatrix}
  \phi\\ \psi\\\dot\phi\\ \dot\psi 
  \end{pmatrix},
F=\begin{pmatrix}
  \xi_3 \\ \xi_4\\M^{-1}f 
  \end{pmatrix},
G=\begin{pmatrix}
  0 \\ 0\\M^{-1}(-Q_N+\mu Q_T) 
  \end{pmatrix}.
\end{align*}
Using the procedure described in Sec. \ref{sec:2} we can derive expressions for $v$, $p$, $b$,
$\alpha_1$, $\alpha_2$, and $\alpha_3$. These are given in Appendix \ref{A:2}.

To study the behaviour near the singularity, we reduce the number of
parameters by setting 
\begin{align*}
  m_2&=m_1,    \qquad \mu =\frac{25}{12}, \quad 
  \phi_0 =\phi^*+\frac{49}{20}-\frac{7}{12}\beta-\frac{9}{50}\kappa, \\
  k_\phi&=m_1gl, \quad  c_\phi =0, \qquad 
  k_\psi=\kappa\frac{m_1g}{l},  \quad 
c_\psi=\frac{25}{108}(18-7\beta)m_1\sqrt{\frac{g}{l}}.
\end{align*}
where the values of $m_1$, $g$ and $l$ determines a scale for mass,
length, and time, but not have any other influence on the dynamics of
the system. We leave the two parameters $\beta$ and $\kappa$ (of
dimension 1) to be specified later.

The chosen values of $\mu$ and $\phi_0$ ensure that that we have the
singularity at
\begin{align*}
  \cos(\phi^*)& = \frac{3}{5}, \quad  \sin(\phi^*) = \frac{4}{5} 
\quad 
  \dot{\phi}^* =-\sqrt{\frac{g}{l}} \\
  \psi^*&= -\frac{2}{5}l \quad \dot{\psi}^* = \frac{4}{5}\sqrt{gl}.
\end{align*} 
Furthermore we have
\begin{align*}
  \alpha_1^*&=-\frac{175}{408}\frac{1}{m_1}\sqrt{\frac{g}{l}}, \qquad 
  \alpha_3^* =-\frac{175}{408}\beta\frac{1}{m_1}\sqrt{\frac{g}{l}} \\
  \alpha_2^*&= \frac{30625\beta^2+9450\beta\kappa-61425\beta+8262\kappa-126765}{55080}g\sqrt{\frac{g}{l}}
\end{align*}
which means that the quotient between $\alpha_3^*$ and $\alpha_1^*$ is
equal to $\beta$ in accordance with \eqref{eq:beta}, and the sign of
$\alpha_2^*$ is controlled by $\kappa$. Hence, the frictional impact
oscillator may belong to any of the classes I, II, and III,
furthermore, $\beta$ may take any desired value.

\subsection{Numerical verification in case III}

For numerical simulations, we use units based on $m_1$, $l$, and
$g$. By taking $\beta=7/3$, $\kappa=0$, we get $\alpha_1^*=-0.4289$,
$\alpha_2^*=-1.8764$, and $\alpha_3^*=-1.0008$, which corresponds to
case III. For contact smoothing, we use the compliant model of Sec. 2 with $\eps=10^{-6}$. Two sets of
initial conditions are tested: The angle coordinate is set to
\begin{equation*}
  \phi=\phi^*+0.1,\quad\dot{\phi}=-0.9\text{ or }\phi=\phi^*+0.1,\quad\dot{\phi}=-0.5
\end{equation*}
and the linear coordinate is set to be on just in contact
$y=v=0$:
\begin{equation*}
  \psi=\cos(\phi)-1,\quad\dot{\psi}=-\sin(\phi)\dot{\phi}.
\end{equation*}
Relaxation of the $z$ dynamics was found to take about $10^{-3}$ time
units, whereas the system was simulated for ${\cal O}(10^{-1})$ time
units.  Additionally, a third initial condition approximately on the
distinguished trajectory at $p=0.01$ was chosen.

\begin{figure}
\begin{center}
\includegraphics[width=0.7\textwidth]{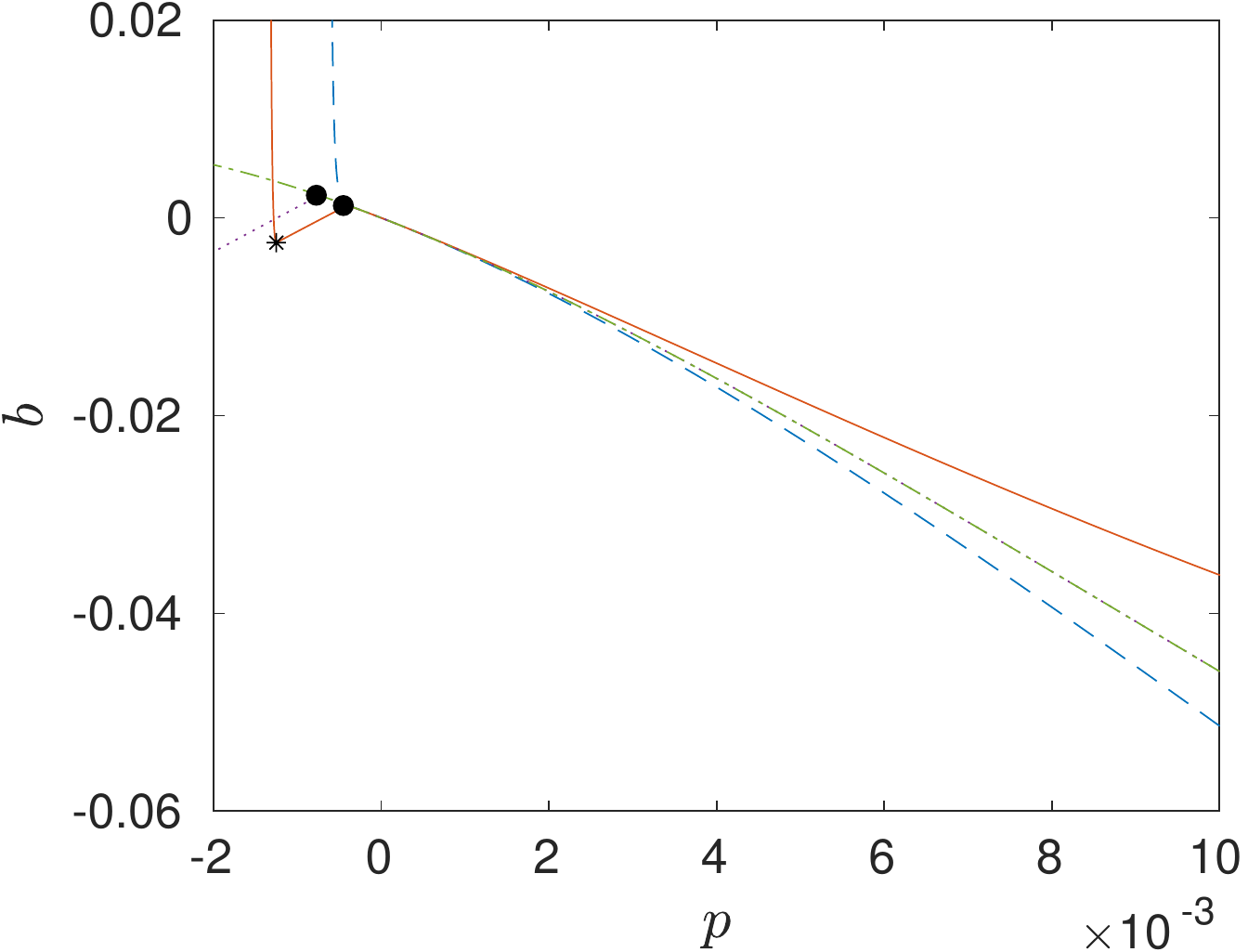}
\includegraphics[width=0.7\textwidth]{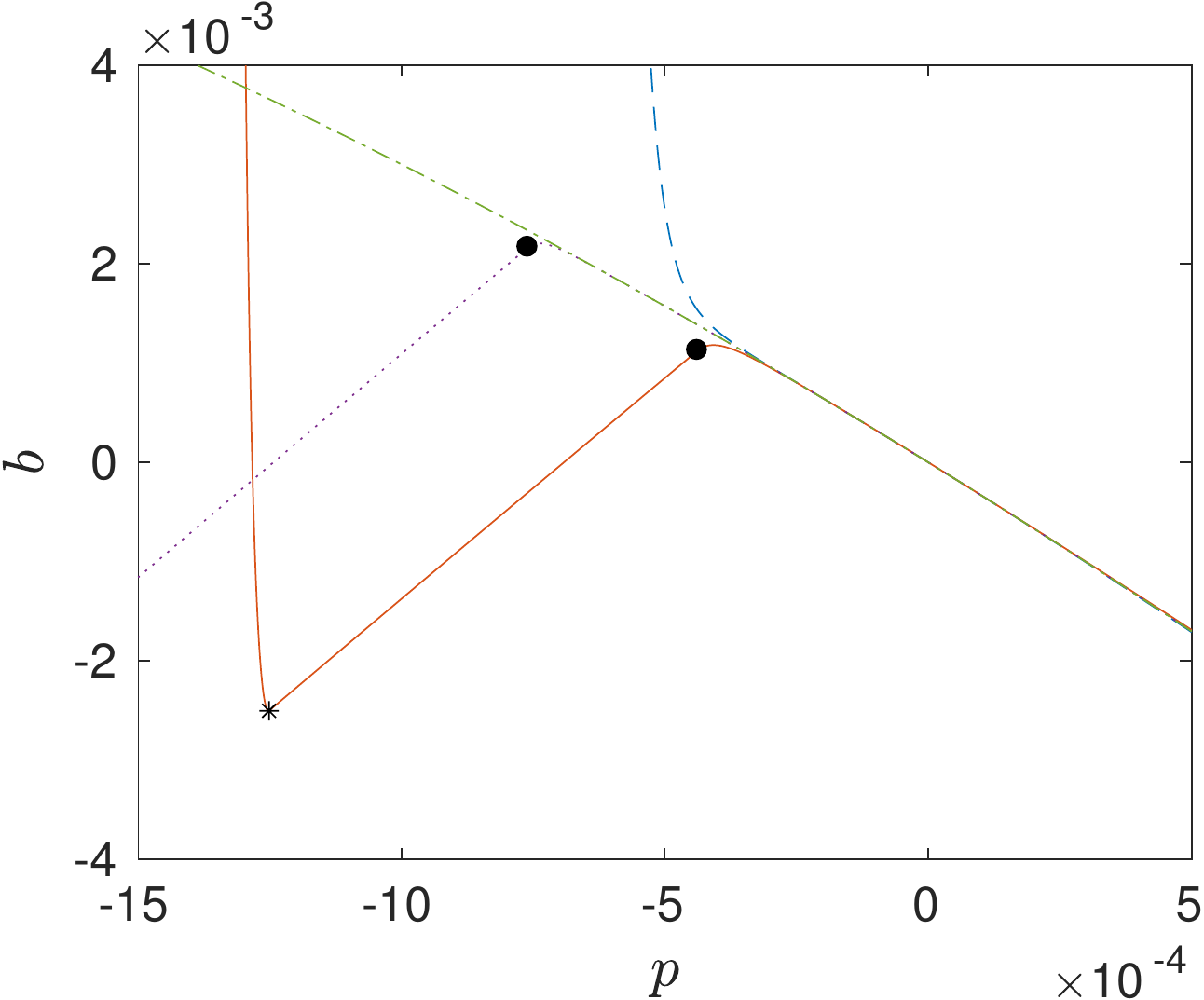}
\caption{Numerical simulations of $p$ versus $b$ for the frictional
  impact oscillator with three different initial conditions
  (red/solid, blue/dashed, and purple/dotted curves), and asymptotic
  approximation of the distinguished trajectory (green/dash-dotted
  curve). Lift-off events are marked with a solid circle symbol, touch-down events with an asterisk symbol. The lower panel  is a zoomed version of the same diagrams.}
\label{fig:LeineFig1}
\end{center}
\end{figure}

\begin{figure}
\begin{center}
\includegraphics[width=0.7\textwidth]{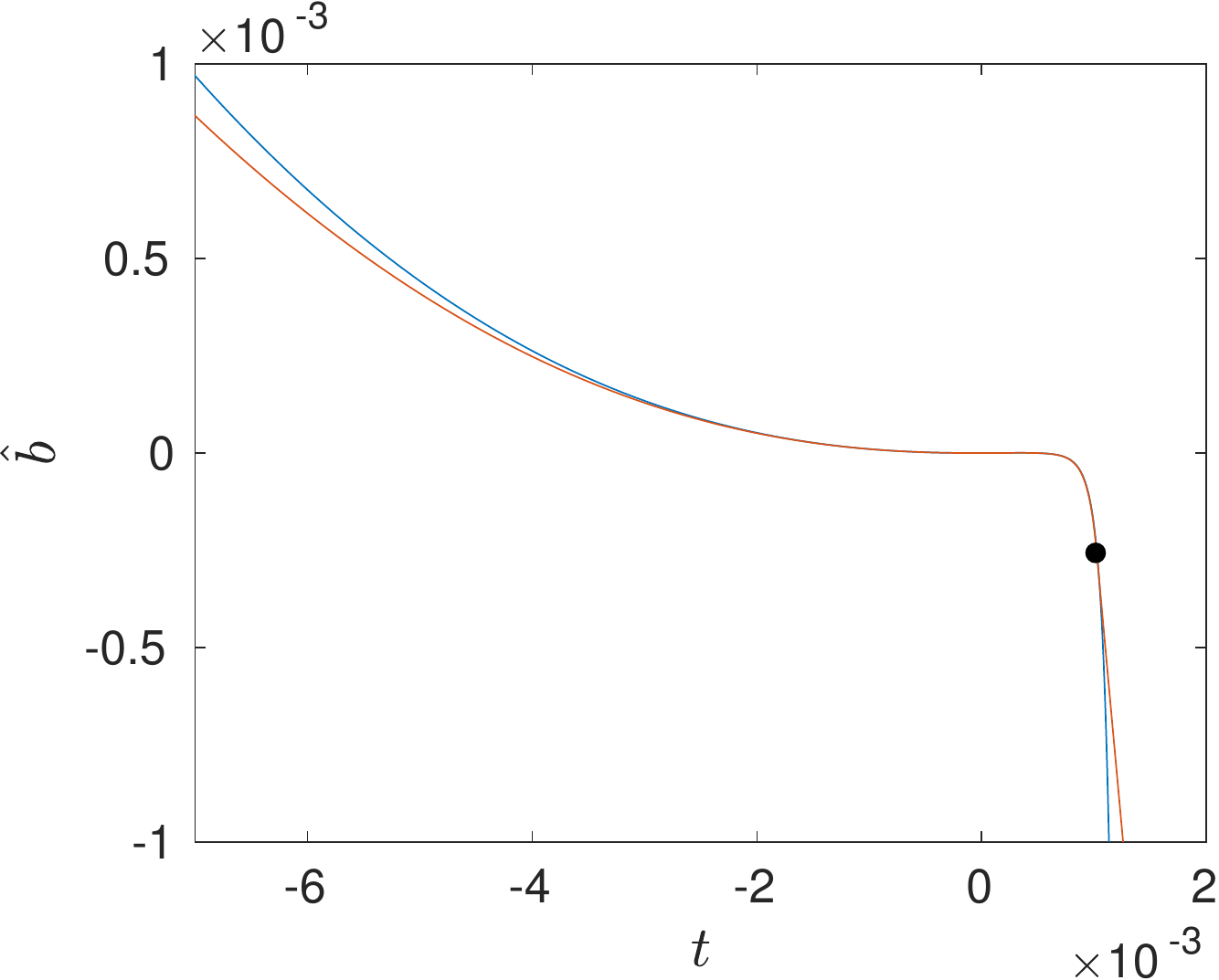}
\includegraphics[width=0.7\textwidth]{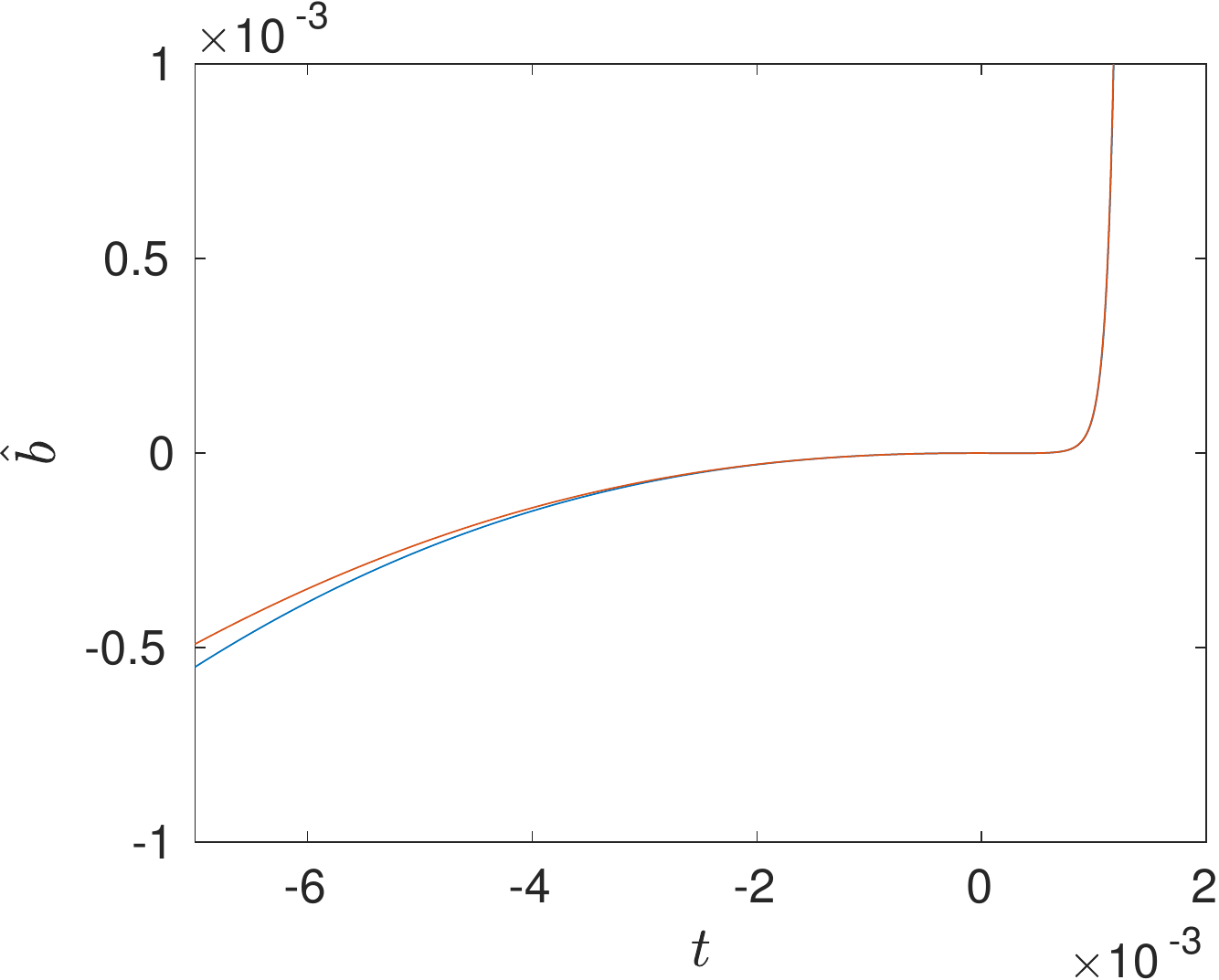}
\caption{Diagrams of $\hat{b}$ versus time of the frictional impact oscillator for two initial conditions. The blue curves were obtained by numerical simulation, whereas the red curves are given by the suitable scaled $\Theta$ function.} 
\label{fig:LeineFig2}
\end{center}
\end{figure}


Figure \ref{fig:LeineFig1} shows a $p$ vs $b$ diagram. The red/solid
curve (first initial condition) passes the (ghost) singularity, lifts
off and then touches down again, initiating an ``impact''. The
blue/dashed curve (second initial condition) goes directly to an
''impact''. The purple/dotted curve is the trajectory using an initial
condition approximately on the distinguished trajectory, and the
green/dash-dotted curve is the appoximate distinguished trajectory
itself, computed from the power series with $M=15$. These two are
indistinguishable for $p>0$, but although the purple/dotted curve is able
to follow the distinguished trajectory further into $p<0$ than the
other initial conditions, it still eventually deviates. In all, these
results are consistent with our finding that the distinguished
trajectory is on a separatarix.

Figure \ref{fig:LeineFig2}(a) shows a diagram of the deviation
$\hat{b}$ (see \eqref{eq:bpert}) from the distinguished trajectory
versus $t$.  The time origin is shifted to make $p=0$ at $t=0$. The
red curve is simulation using the first initial condition. The blue
curve is computed using the suitably scaled hypergeometric $\Theta$
function, where time scale is based on $\alpha_1^*$, and the amplitude
scale is adjusted to make the curves coincide when $t=0$. Lift-off in
the simulation takes place just after $t=10^{-3}$, explaining the
fast-growing deviation between the two curves for more positive times,
since the hypergeometric solution assumes contact. At the same time,
the deviation for negative times grows more slowly, and it is a
natural consequence of the approximations used when developing the
inner system. Figure \ref{fig:LeineFig2}(b) shows the same thing for
the second initial condition. In this case there is no loss of
contact, and the two curves fit each other very well for positive
times.

\section{Conclusion}
\label{sec:7}

The analysis in this paper provides a key step in the resolution of
one of the simplest consequences of the paradox on the inconsistency
of rigid body mechanics subject to Coulomb frication, first described
by Painlev\'{e} in 1895 \cite{Painleve1895}.  Despite numerous treatments
in the intervening 120 years or so, as pointed out in
\cite{PainleveReview}, there remain many unsolved problems. Even for
planar configurations with a single frictional point contact, 
it was previsouly known that open sets of initial conditions can 
approach the finite-time singularity that is known as dynamic jam, 
represented by the G-spot.  
What we have established in
this paper is a general method for establishing what happens
beyond the G-spot, at least in theory, and also understanding the
sensitivity of what is observed to any smoothing through contact
regularisation.

There are several weakness to the analysis we have presented. First, we 
have been unable to resolve in general what happens beyond the first lift-off
or onset of IWC. Not only is there extreme sensitivity because during an 
IWC, but lift-off occurs with vanishingly small free normal acceleration 
as $\eps \to 0$. In cases III and I this would occur with 
$\dot{b}<0$ so that lift-off would lead rapidly to further impact with 
small normal velocity. 
Whether this impact would again lead to further lift off close to the G-spot 
is unclear in general. It is conceivable that in the limit $\eps\to 0$ one
might have an infinite sequence of impacts with which might accumulate either
in forward time (chatter) or in reverse time (reverse chatter). The latter
would represent a point of infinite indeterminancy, as analysed in
\cite{paper2}. Further analysis of the dynamics post the first lift-off will
form the subject of future work. 

A second weakness is a lack of rigour. While we have 
formulated the existence of the distinguished trajectory as a Theorem,
in general our analysis is asymptotic in nature. There is also a frustrating
lack of a proof in cases where we have identified that an IWC probably occurs, 
because we cannot rule out the possibility of a lift-off in certain 
pathological examples. In particular, even though the
asymptotics indicate a trajectory for which $\tilde{\hat{y}}$ diverges to 
$-\infty$ for large $s>0$ and $\tilde{\hat{y}} \ll 0$ for $s=0$, this 
is not sufficient to show that $\tilde{\hat{y}}$ remains negative 
for all $s>0$. 
Numerical results indicate that an impact always occurs. Perhaps further
study of the appropriate generalised hypergeometric functions will 
shed further light on this question.
During the final preparation of this manuscript we also become aware
of the independent work of Hogan \& Kristiansen \cite{HoganKristiansen2}
which studies a similar problem to the one considered here. They use
completely different methods, namely geometric singular perturbation theory,
to establish the existence of a canard trajectory. It is probable that 
a combination of their analysis with the asympotitic analysis conducted here
would lead to some more comprehensive results. 

A third weakness is the lack of experimental
work to confirm what might happen in practice. In fact,
while there have been several practical observations of the
consequences of the Painlev\'{e} paradox (see \cite{PainleveReview}),
we are not aware of any detailed quantitative experimental studies.
One of the difficulties here is that dynamic jam represents a point of
extreme sensitivity in the dynamics, therefore what is observed is
likely to be highly dependent on the precise details of any
imperfections or asperities in any practical model. Nevertheless, it
would seem to be high time for the design of a detailed test rig to
demonstrate each of cases I to III illustrated here.

Finally we should point out that the problem studied here is rather 
idealised. In practice, no structure ever undergoes point contact per se,
there is always some form of regional contact. As shown in \cite{Varkonyi_unp}
the dynamics of systems with multiple point contacts can be much more complex,
with various novel forms of Painlev\'{e} paradox that involve interaction
between simultaneous contacts. Also, as demonstrated in 
\cite{PainleveReview}[Sec.~7], there is yet more complexity 
if we study fully three-dimensional dynamics. For example, for
certain configurations it is possible to enter 
the Painlev\'{e} region $p<0$ without passing through a neighbourhood of
the G-spot. 

There are clearly many situations that require further analysis
along the lines developed in this paper. 

\subsection*{Acknowledgements}
This work was initiated at the Centre Recherca Matem\`{a}tica 
(CRM) Barcelona during the three-month programme in 2016 on Nonsmooth
Dynamical Systems. The authors thank the CRM for its support, and
especially Mike Jeffrey and
Thibaut Putelat for useful discussion. Preliminary ideas for this paper
were developed in collaboration with Harry Dankowicz, whose insights we also
gratefully acknowledge.  
We are also grateful to Kristian Kristiansen and John Hogan for sharing
their unpublished independent work with us at the latter stages of
preparation of this paper.  PLV acknowledges support from the
National Research, Innovation and Development Office of Hungary under
grant K104501 and ARC from the UK EPSRC under Programme Grant
``Engineering Nonlinearity'' EP/K003836/2. 

\bibliographystyle{plain}
\bibliography{painleve}

\newpage

\appendix
\numberwithin{equation}{section}
\section{Generalised hypergeometric functions and their large time 
asymptotics}
\label{A:1}

Consider the following third-order non-autonomous equation
\begin{equation}
  \frac{d^3}{d \tau^3}\theta-\tau\frac{d}{ d\tau} \theta+\beta\theta=0.
\label{eq:btau}
\end{equation}
The solutions to this equation can be expressed in terms of generalised
hypergeometric functions.  In particular, by standard
results, see e.g.~\cite[Ch.16]{NIST}, we have the following result.

\begin{theorem}
\label{th:3}
The general solution of the 
differential equation \eqref{eq:btau} 
can be expressed as
\begin{equation*}
  \theta(\tau)=\theta(0)\:{}_1F_2\left(-\frac{\beta}{3};\frac{1}{3},\frac{2}{3};\frac{\tau^3}{9}\right)+
  \frac{d}{d\tau}\theta(0)\tau\:{}_1F_2\left(\frac{1}{3}-\frac{\beta}{3};\frac{2}{3},\frac{4}{3};\frac{\tau^3}{9}\right)+
  \frac{d^2}{d\tau^2}\theta(0)\frac{\tau^2}{2}\:{}_1F_2\left(\frac{2}{3}-\frac{\beta}{3};\frac{4}{3},\frac{5}{3};\frac{\tau^3}{9}\right).
\end{equation*}
where ${}_1F_2$ is the generalised hypergeometric function
with indices $[1,2]$.  
\end{theorem}

\bigskip

We are interested in the asympotics of this solution as $|\tau| \to \infty$. 
Using the general asymptotic expansion of
${}_1F_2$ for complex arguments in the limit of large $|\tau|$, we can formulate
the following results. 

\bigskip

\begin{theorem}[Asymptotics of ${}_1F_2$ for large negative $\tau$]
\label{th:4}
Define $h$ as a formal series
\begin{equation}
  h(\tau,\beta)=3^{\beta/3-1}(-\tau)^\beta\sum_{k=0}^\infty\frac{1}{k!3^k(-\tau)^{3k}\Gamma(1+\beta-3k)}
\label{eq:hdef}
\end{equation}
and $e_r$, $e_i$ as the real and imaginary
parts of the formal series
\begin{equation}
  e_r(\tau,\beta)+ie_i(\tau,\beta)=\sqrt{\pi}3^{\beta/3-1}(-\tau)^{-\beta/2-3/4}e^{i(\pi\beta/6+\pi/4-2(-\tau)^{3/2}/3)}\sum_{k=0}^\infty
  \frac{c_k(\beta)i^k3^k}{2^k(-\tau)^{3k/2}}.
\label{eq:eier}
\end{equation}
Here the coefficients $c_k$ are determined via a somewhat complicated 
recurrence relation, see \cite[Eq.16.11.4]{NIST}. 
In particular, we have 
\begin{equation}
c_0=1. 
\label{eq:cisdef}
\end{equation}
Then, provided $\beta$ is not an integer, asymptotically, as $\tau\rightarrow-\infty$
\begin{equation}
  {}_1F_2\left(-\frac{\beta}{3};\frac{1}{3},\frac{2}{3};\frac{\tau^3}{9}\right)
  \sim3\Gamma\left(\frac{3+\beta}{3}\right)h(\tau,\beta)
  +\frac{3}{\Gamma\left(-\frac{\beta}{3}\right)}2e_r(\tau,\beta),
\label{eq:bsmall1}
\end{equation}
\begin{equation}
  \tau\:{}_1F_2\left(\frac{1}{3}-\frac{\beta}{3};\frac{2}{3},\frac{4}{3};\frac{\tau^3}{9}\right),
  \sim-3^{2/3}\Gamma\left(\frac{2+\beta}{3}\right)h(\tau,\beta)
  -\frac{3^{2/3}}{\Gamma\left(\frac{1-\beta}{3}\right)}\left(e_r(\tau,\beta)-\sqrt{3}e_i(\tau,\beta)\right),
\label{eq:bsmall2}
\end{equation}
\begin{equation}
  \frac{\tau^2}{2}\:{}_1F_2\left(\frac{2}{3}-\frac{\beta}{3};\frac{4}{3},\frac{5}{3};\frac{\tau^3}{9}\right),
  \sim3^{1/3}\Gamma\left(\frac{1+\beta}{3}\right)h(\tau,\beta)
  +\frac{3^{1/3}}{\Gamma\left(\frac{2-\beta}{3}\right)}\left(-e_r(\tau,\beta)-\sqrt{3}e_i(\tau,\beta)\right).
\label{eq:bsmall3}
\end{equation}
\end{theorem}

\bigskip

We want to choose the specific solution 
$\theta(\tau)=\Theta(\tau)$ whose initial conditions are such that 
the coefficients 
of the highly oscillatory terms $e_r$ and $e_i$ vanish. 
The remaining $h$ term is dominated by its first term, 
which is proportional to $(-\tau)^\beta$. 
In particular, using the particular initial conditions
\begin{eqnarray}
  \Theta(0)&=&\frac{\Gamma\left(\frac{-\beta}{3}\right)}{3^{(3+\beta)/3}\Gamma(-\beta)}, \label{eq:t01}  \\
  \frac{d\Theta}{d\tau}(0)&=&\frac{\Gamma\left(\frac{1-\beta}{3}\right)}{3^{(2+\beta)/3}\Gamma(-\beta)},  \label{eq:t02} \\
  \frac{d^2\Theta}{d \tau^2}(0)&=&\frac{\Gamma\left(\frac{2-\beta}{3}\right)}{3^{(1+\beta)/3}\Gamma(-\beta)}, \label{eq:t03}
\end{eqnarray}
we define a function $\Theta$ with the asymptotic behaviour
\begin{equation*}
  \Theta(\tau,\beta)\sim3^{1-\beta/3}\Gamma\left(1+\beta\right)h(\tau,\beta)\sim(-\tau)^\beta
\end{equation*}
as $\tau\rightarrow-\infty$.

\bigskip

\begin{theorem}[Asymptotic of ${}_1F_2$ for large positive $\tau$]
\label{th:5}
Define $e_1$ as the formal series
\begin{equation*}
  e_1(\tau,\beta)=\sqrt{\pi}3^{\beta/3-1}(\tau)^{-\beta/2-3/4}e^{2\tau^{3/2}/3}\sum_{k=0}^\infty
  \frac{c_k(\beta)3^k}{2^k(\tau)^{3k/2}}.
\end{equation*}
where the $c_k$ coefficients are defined as in the previous theorem.
Then asymptotically, as $\tau\rightarrow\infty$,
\begin{equation}
  {}_1F_2\left(-\frac{\beta}{3};\frac{1}{3},\frac{2}{3};\frac{\tau^3}{9}\right)
  \sim\frac{3}{\Gamma\left(-\frac{\beta}{3}\right)}e_1(\tau,\beta)
\label{eq:blarge1}
\end{equation}
\begin{equation}
  \tau\:{}_1F_2\left(\frac{1}{3}-\frac{\beta}{3};\frac{2}{3},\frac{4}{3};\frac{\tau^3}{9}\right)
  \sim\frac{3^{2/3}}{\Gamma\left(\frac{1-\beta}{3}\right)}e_1(\tau,\beta)
\label{eq:blarge2}
\end{equation}
\begin{equation}
  \frac{\tau^2}{2}\:{}_1F_2\left(\frac{2}{3}-\frac{\beta}{3};\frac{4}{3},\frac{5}{3};\frac{\tau^3}{9}\right)
  \sim\frac{3^{1/3}}{\Gamma\left(\frac{2-\beta}{3}\right)}e_1(\tau,\beta)
\label{eq:blarge3}
\end{equation}
\end{theorem}

\bigskip

Applied to the $\Theta$ functions this means
\begin{equation}
  \Theta(\tau,\beta)\sim \frac{3^{1-\beta/3}}{\Gamma(-\beta)}e_1(\tau,\beta)
\sim
\frac{\sqrt{\pi}}{\Gamma(-\beta)} \exp((2/3) \tau^{3/2})(\tau)^{-\beta/2-3/4} 
 \label{eq:tlarge}
\end{equation}
as $\tau\rightarrow\infty$, where we have used $c_0=1$.

\section{Expressions for the frictional impact oscillator}
\label{A:2}

\begin{equation*}
  v(\phi,\psi,\dot{\phi},\dot{\psi})=\dot{\psi}+l\sin(\phi)\dot{\phi}
\end{equation*}
\begin{equation*}
  p(\phi,\psi)=\frac{m_1\cos(\phi)^2+m_2\sin(\phi)(\sin(\phi)-\mu\cos(\phi))}{m_1(m_2+m_1\cos(\phi)^2)}
\end{equation*}
\begin{equation*}
  b(\phi,\psi,\dot{\phi},\dot{\psi})=
  \frac{m_1l\cos(\phi)\left(m_2l\dot{\phi}^2-\cos(\phi)(k_\psi\psi+c_\psi\dot{\psi})\right)-m_2\sin(\phi)
    (k_\phi(\phi-\phi_0)+c_\phi\dot{\phi})}
  {m_1l(m_2+m_1\cos(\phi)^2)}-g
\end{equation*}
\begin{equation*}
  \alpha_1(\phi,\psi,\dot{\phi},\dot{\psi})=
  -\frac {m_2\left(\mu\left(\cos(\phi)^2\left(m_1+2m_2\right)-m_2\right)-2m_2\sin(\phi)\cos(\phi)\right)}{m_1 \left(m_2+m_1\cos(\phi)^2 \right)^2}\dot{\phi}
\end{equation*}
\begin{multline*}
  \alpha_2(\phi,\psi,\dot{\phi},\dot{\psi})=
  \frac{1}{m_1^2l^3\left( m_2+m_1\cos(\phi)^2\right)^2}
  \left[
    m_1^2\left(m_2+m_1\cos(\phi)^2\right){l}^{3}gc_{\psi}\cos(\phi)^2
    \right.\\
    -m_1^2\left( m_2+m_1\cos(\phi)^2\right){l}^{3}k_{\psi}\cos(\phi)^2\dot{\psi}
    +m_1^{3}{l}^{4}c_{\psi}\cos(\phi)^{3}{\dot{\phi}}^{2}
    -m_1^{2}{l}^{2}c_{\psi}\cos(\phi)^{2}\sin(\phi)\left(k_{\phi}\left(\phi-\phi_0\right)+c_{\phi}\dot{\phi}\right)\\
    +m_1^{2}{l}^{3}c_{\psi}\cos(\phi)^{2}\left(k_{\psi}\psi+c_{\psi}\dot{\psi}\right)
    -m_1m_2lc_{\phi}\sin(\phi)^{2}\left(k_{\psi}\psi+c_{\psi}\dot{\psi}\right)\\
    -m_1^{2}m_2{l}^{2}c_{\phi}\cos(\phi)\sin(\phi)^2\dot{\phi}^{2}
    -m_1^{2}m_2^{2}{l}^{4}\sin(\phi)\dot{\phi}^{3}
    +m_2\left(m_1+m_2\right)c_{\phi}\sin(\phi)\left(k_{\phi}\left(\phi-\phi_0\right)+c_{\phi}\dot{\phi}\right)\\
    +4m_1^{2}m_2{l}^{3}\cos(\phi)\sin(\phi)\dot{\phi}\left(k_{\psi}\psi+c_{\psi}\dot{\psi}\right)
    -3m_1m_2^{2}{l}^{2}\cos(\phi)\dot{\phi}\left(k_{\phi}\left(\phi-\phi_0\right)+c_{\phi}\dot{\phi}\right)\\
    +3m_1^{3}m_2{l}^{4}\cos(\phi)^{2}\sin(\phi)\dot{\phi}^{3}
    +m_1^{2}m_2{l}^{2}\cos(\phi)^{3}\dot{\phi}\left(k_{\phi}\left(\phi-\phi_0\right))+c_{\phi}\dot{\phi}\right)\\
    \left.
    -m_1m_2\left(m_2+m_1\cos(\phi)^2\right){l}^{2}k_{\phi}\sin(\phi)\dot{\phi}
    -4m_1^{2}m_2{l}^{2}\cos(\phi)\dot{\phi}\left(k_{\phi}\left(\phi-\phi_0\right)+c_{\phi}\dot{\phi}\right)
    \right]
\end{multline*}
\begin{multline*}
  \alpha_3(\phi,\psi,\dot{\phi},\dot{\psi})=
  \frac{1}{m_1^2l^2\left(m_2+m_1\cos(\phi)^2\right)^2}
  \left[
  m_1^2l^2c_\psi\cos(\phi)^3\left(\mu\sin(\phi)+\cos(\phi)\right)
  \right.\\
  +2m_1^2m_2l^2\mu\cos(\phi)^2\dot{\phi}
  +2m_1m_2^2l^2\cos(\phi)\left(\mu\cos(\phi)-\sin(\phi)\right)\dot{\phi}\\
  \left.
  -m_2(m_1+m_2)c_\phi \mu\cos(\phi)\sin(\phi)
  +m_2^2c_\phi\sin(\phi)^2
  \right]
\end{multline*}

\end{document}